\documentclass[11pt]{article}

\usepackage{amssymb,amsthm,nccmath,intcalc,cases,enumitem,setspace,float}
\usepackage[dvipsnames]{xcolor}
\usepackage{tikz}
\usepackage[hidelinks]{hyperref}
\usetikzlibrary{calc}
\newtheorem{thm}{Theorem}
\newtheorem{lem}{Lemma}[section]
\newtheorem{prop}[lem]{Proposition}
\newtheorem{cor}[lem]{Corollary}

\newtheorem*{question}{Question}
\theoremstyle{definition}

\newcommand{\Z}{\mathbb{Z}}
\def \mod#1{{\:({\rm mod}\ #1)}}

\newcommand{\cms}{\textnormal{cms}}
\newcommand{\ms}{\textnormal{ms}}
\renewcommand{\geq}{\geqslant}
\renewcommand{\leq}{\leqslant}
\renewcommand{\ge}{\geqslant}

\let\oldproofname=\proofname
\renewcommand{\proofname}{\rm\bf{\oldproofname}}

\title{The cyclic matching sequenceability of regular graphs}
\author{Daniel Horsley
\thanks{
School of Mathematics, Monash University, Clayton 3800, Australia.}\\
\texttt{danhorsley@gmail.com}\\
\and
Adam Mammoliti $^{*,}$
\thanks{Corresponding Author.}\\
\texttt{adam.mammoliti@outlook.com.au}
}
\date{}

\begin{document}
\setstretch{1.1}
\maketitle
\begin{abstract}
The {\em cyclic matching sequenceability} of a simple graph $G$, denoted $\cms(G)$, is the largest integer $s$ for which there exists a cyclic ordering of the edges of $G$ so that every set of $s$ consecutive edges forms a matching. In this paper we consider the minimum cyclic matching sequenceability of $k$-regular graphs. We completely determine this for $2$-regular graphs, and give bounds for $k \geq 3$.
\end{abstract}

\section{Introduction}

The {\em cyclic matching sequenceability} of a simple graph $G$,
denoted $\cms(G)$,
is the largest integer $s$ for which
there exists a cyclic ordering of the edges of $G$ so that every set of $s$ consecutive edges forms a matching.
Katona~\cite{MR2181045} implicitly considered cyclic matching sequenceability
and found a lower bound for $\cms(K_n)$.
Brualdi, Kiernan, Meyer and Schroeder~\cite{MR2961987}
defined cyclic matching sequenceability explicitly
and proved that $\cms(K_n) = \left\lfloor\frac{n-2}{2}\right\rfloor$
for all $n\geq 4$,
thus strengthening the result found by Katona~\cite{MR2181045}.
Brualdi et al.~\cite{MR2961987} also determined
that $\cms(C_n) = \left\lfloor \frac{n-1}{2} \right\rfloor$
for all $n\geq 3$ and found the cyclic matching sequenceability of several other graphs.

A non-cyclic variant to cyclic matching sequenceability, denoted $\ms(G)$ has also been considered and was defined first by Alspach~\cite{MR2394738} who determined $\ms(K_n)$.
Brualdi et al.~\cite{MR2961987} also determined the matching sequenceability for cycles and several other classes of graphs.
Chiba and Nakano~\cite{ChiNa}
found various results concerning the matching sequenceability for general graphs and more refined results for regular graphs.

In this paper our focus is on the cyclic matching sequenceability of regular graphs. The \emph{chromatic index} of a graph is the smallest number of colours required to properly colour its edges. By Vizing's theorem \cite{Viz} the chromatic index of a graph with maximum degree $\Delta$ is equal to either $\Delta$ or $\Delta+1$. In the former case we say the graph is \emph{class 1} and in the latter we say it is \emph{class 2}. For positive integers $n$ and $k$ such that $n>k$ and $nk$ is even, we define $\cms(n,k)$ to be the minimum value of $\cms(G)$ over all $k$-regular graphs on $n$ vertices and define $\cms_1(n,k)$ to be the minimum value of $\cms(G)$ over all $k$-regular class 1 graphs on $n$ vertices. Our primary focus is on the behaviour of $\cms(n,k)$ and $\cms_1(n,k)$ for fixed $k$ as $n$ becomes large. All asymptotic notation used in this paper is relative to this regime.

The main contribution of this paper is to establish lower bounds on $\cms(n,k)$ and $\cms_1(n,k)$, as well as an upper bound on $\cms(n,k)$.
These bounds are summarised in Theorem~\ref{thm:k reg}.

\begin{thm}\label{thm:k reg}
Let $k \geq 3$ be an integer. Then, for any integer $n \geq 6(k+1)$ such that $nk$ is even,
\begin{alignat*}{2}
  \max\left\{\mfrac{k(5k-3)}{4(k+1)(4k-3)}n-6,\mfrac{31k}{98(k+1)}n-o(n)\right\}  &\leq \cms(n,k) &&\leq \mfrac{kn}{2(k+1)} \\
  \max\left\{\mfrac{5k-8}{4(4k-7)}n-6,\mfrac{31}{98}n-o(n)\right\}  &\leq \cms_1(n,k) &&\leq \mfrac{n-1}{2}\,.
\end{alignat*}
\end{thm}

In the lower bound on $\cms(n,k)$, for large $n$, the $\max$ takes the first value for $k \leq 13$ and the second value for $k \geq 14$. In the lower bound on $\cms_1(n,k)$, for large $n$, the $\max$ takes the first value for $k \leq 14$ and the second value for $k \geq 15$. Table~\ref{T:kconsequences} in the conclusion gives an explicit listing of the consequences of Theorem~\ref{thm:k reg} for various values of $k$. In the case of 2-regular graphs we are able to completely determine $\cms(n,2)$ and $\cms_1(n,2)$.

\begin{thm}\label{thm:2 reg}
For each $n \geq 6$, we have $\cms(n,2)=\lfloor\frac{n}{3}\rfloor$ and for each even $n \geq 4$, we have $\cms_1(n,2)=\frac{n-2}{2}$.
\end{thm}

For a graph $G$, let $\mathcal{M}(G)$ be the set of all matchings in $G$ and, for an edge $e$ of $G$, let $\mathcal{M}_e(G)$ be the set of all matchings in $G$ containing $e$. A \emph{fractional edge colouring} of a graph $G$ is a function $\omega: \mathcal{M}(G) \rightarrow \mathbb{R}^{\geq 0}$ such that $\sum_{M \in \mathcal{M}_e(G)}\omega(M) \geq 1$ for each edge $e \in E(G)$. The weight of such a colouring is $\sum_{M \in \mathcal{M}(G)}\omega(M)$. Note that an edge colouring of $G$ can be viewed as a fractional edge colouring $\omega$ of $G$ for which the image of $\omega$ is a subset of $\{0,1\}$. The \emph{fractional chromatic index} of a graph $G$ is the infimum of the weights of the fractional edge colourings of $G$. While the main focus of this paper is regular graphs, some of our results apply more generally. In particular, we have the following.

\begin{thm}\label{thm:gen bounds}
For any graph $G$ with chromatic index $c$ and fractional chromatic index $c_f$,
\[\left\lfloor\tfrac{1}{2c}|E(G)|\right\rfloor-1 \leq \cms(G) \leq \tfrac{1}{c_f}|E(G)|.\]
Furthermore, for any integers $\Delta \geq 2$ and $n \geq \Delta+1$, there is a graph $G$ of order $n$ such that $\cms(G) \leq \frac{1}{\Delta+1}|E(G)|$.
\end{thm}

The lower bound in Theorem~\ref{thm:gen bounds} differs by at most two from an analogous bound for matching sequenceability given in \cite{ChiNa}.

We organise the rest of the paper as follows. In Section~\ref{Preliminaries} we introduce some notation and preliminary results. In Section~\ref{2reg} we consider $2$-regular graphs and prove Theorem~\ref{thm:2 reg}. Theorem~\ref{thm:gen bounds} is proved in Section~\ref{genBounds}. In Section~\ref{givenPartition} we establish a lower bound on the cyclic matching sequenceability of a regular graph assuming the existence of a partition of its edges with suitable properties. Finally in Section~\ref{findPartition}, we show that regular graphs do admit such partitions and prove Theorem~\ref{thm:k reg}.

\section{Preliminaries}
\label{Preliminaries}

For an integer $n$, let $\Z_n$ represent the additive group of integers modulo $n$. In this paper, graphs will always be simple. Two edges in a graph are \emph{adjacent} if they are both incident
on the same vertex. A \emph{matching} is a $1$-regular graph. The \emph{union} $G \cup H$ of two graphs $G$ and $H$ is the graph with vertex set $V(G) \cup V(H)$ and edge set $E(G) \cup E(H)$. An {\em ordering} of a graph $G$ with $m$ edges is a bijective function
$\ell \;:\; E(G) \rightarrow \Z_{m}$.
The image of $e$ under~$\ell$ is called the {\em label} of $e$. We will sometimes specify an ordering $\ell$ by giving the tuple $(\ell^{-1}(0),\ldots,\ell^{-1}(m-1))$ rather than the function $\ell$. A set of edges of $G$ is {\em consecutive} in $\ell$ if their labels form a set of consecutive integers
and is {\em cyclically consecutive} in $\ell$ if their labels form a set of consecutive integers
modulo $m$.

Let $\ell$ be an ordering of a graph $G$ with $m$ edges and let $e$ and $e'$ be distinct edges of $G$. We define $d_\ell(e,e')$, the \emph{forward distance from $e$ to $e'$ in $\ell$}, to be the smallest positive integer $d$ such that $\ell(e)+d = \ell(e')$, where the addition takes place in $\Z_{m}$. We define $d_\ell\{e,e'\}$, the \emph{distance between $e$ and $e'$ in $\ell$}, to be $\min\{d_\ell(e,e'),d_\ell(e',e)\}$. Define $\cms(\ell)$ to be the largest element $s$ of $\{1,\ldots,m\}$ such that $d_\ell\{e,e'\} \geq s$ for any pair $\{e,e'\}$ of edges adjacent in $G$. Similarly, define $\ms(\ell)$ to be the largest element $s$ of $\{1,\ldots,m\}$ such that $d_\ell(e,e') \geq s$ for any ordered pair $(e,e')$ of edges adjacent in $G$ such that $\ell(e)<\ell(e')$.  Note that, for a graph $G$, $\ms(G)$ and $\cms(G)$, as defined in the introduction, are the maximum values of $\ms(\ell)$ and $\cms(\ell)$ respectively over all orderings $\ell$ of $G$. If $G$ is a matching, then obviously $\cms(G)= \ms(G)=|E(G)|$.

We first prove the upper bound of Theorem~\ref{thm:gen bounds}. To our knowledge this connection between the cyclic matching sequenceability of a graph and its fractional chromatic index has not been observed before.

\begin{lem}\label{lem:fracChromatic}
For any graph $G$ with fractional chromatic index $c_f$, $\cms(G) \leq \frac{1}{c_f}|E(G)|$.
\end{lem}

\begin{proof}
Let $s=\cms(G)$ and let $\ell$ be an ordering of $G$ with $\cms(\ell)=s$. Let $\mathcal{L}$ be the set of matchings in $G$ whose edges form a set of $s$ cyclically consecutive edges in $\ell$. Let $\omega: \mathcal{M}(G) \rightarrow \mathbb{R}^{\geq 0}$ be defined by $\omega(M)=\frac{1}{s}$ if $M \in \mathcal{L}$ and $\omega(M)=0$ otherwise. Then $\omega$ is a fractional edge colouring of $G$ with weight $\frac{1}{s}|E(G)|$. So $\frac{1}{s}|E(G)| \geq c_f$ and the result follows.
\end{proof}

For edge-disjoint graphs $G_0$ and $G_{1}$,
with labellings $\ell_{0}$ and $\ell_1$ respectively,
let $\ell_0 \vee \ell_{1}$ denote
the ordering $\ell$ of $G= G_0 \cup G_1$
defined by
$\ell(e) = \ell_0 (e)$ if $e \in E(G_0)$ and $\ell(e) = |E(G_0)|+\ell_1 (e)$ if $e \in E(G_1)$.
A {\em matching decomposition} of a graph $G$ is
a set of edge-disjoint matchings of $G$ that partition the edge set of~$G$. A matching decomposition of $G$ into $k$ matchings can also be viewed as a proper edge colouring of $G$ with $k$ colours.
Now we will provide a lower bound on $\cms(G)$, given a matching decomposition of $G$
with certain properties exists, in the form of the proposition below.
Similar results were implicitly used by Alspach~\cite{MR2394738} and Brualdi et. al.~\cite{MR2961987}.

\begin{prop}[\cite{Ma}]\label{prop: Matching decomposition}
Let $G$ be a graph that decomposes into matchings $M_0 ,\ldots , M_{t-1}$, each with at least $m$ edges
and orderings $\ell_0 , \ldots , \ell_{t-1}$, respectively.
If, for some $s \in  \{1,\ldots,m\}$, $\ms(\ell_i  \vee \ell_{i+1}) \geq s$ for all $i\in \Z_t$, then $\cms(G) \geq s$.
\end{prop}

\begin{proof}
Let $\ell= \bigvee_{i=0}^{t-1} \ell_i$.
Consider two distinct edges $e$ and $e'$ that are at distance less than $s$ in $\ell$. Then $e,e' \in E(M_i \cup M_{i+1})$ for some $i \in \Z_t$.
So, by the assumption that $\ms(\ell_i  \vee \ell_{i+1}) \geq s$, $e$ and $e'$ are nonadjacent in $G$.
This proves the proposition.
\end{proof}

We now prove four further lemmas which, like Proposition~\ref{prop: Matching decomposition}, provide lower bounds on the matching sequenceability of concatenations of orderings under various conditions.

\begin{lem}\label{lem:adding three orderings}
Let $X$, $Y$ and $Z$ be edge-disjoint graphs with orderings $\ell_X$, $\ell_Y$ and $\ell_Z$, respectively. Then
\[
\ms(\ell_X \vee \ell_Y \vee \ell_Z) \geq \min\left\lbrace \ms(\ell_X \vee \ell_Y),\ms(\ell_Y \vee \ell_Z), |E(Y)|+ \ms(\ell_X \vee \ell_Z) \right\rbrace.
\]
\end{lem}
\begin{proof}
Let $G=X \cup Y \cup Z$ and $\ell = \ell_X \vee \ell_Y \vee \ell_Z$. Let $e$ and $e'$ be a pair of adjacent edges in $G$ with $\ell(e)<\ell(e')$. If $e,e' \in E(X \cup Y)$, then $d_\ell(e,e') \geq  \ms(\ell_X \vee \ell_Y)$, by definition. If $e,e' \in E(Y \cup Z)$, then $d_\ell(e,e') \geq  \ms(\ell_Y \vee \ell_Z)$, by definition. Otherwise, $e \in E(X)$ and $e' \in E(Z)$, so $d_{\ell_X \vee \ell_Z}(e,e') \geq  \ms(\ell_X \vee \ell_Z)$ by definition, and hence $d_{\ell}(e,e') \geq  |E(Y)|+\ms(\ell_X \vee \ell_Z)$.
\end{proof}

\begin{lem}\label{lem:adding 4 special orderings}
Let $M_0,M_1,M_2,M_3$ be edge-disjoint matchings of sizes $m_0,m_1,m_2,m_3$ such that $M_0 \cup M_1$, $M_1 \cup M_2$ and $M_2 \cup M_3$
are also matchings. Then, for any orderings $\ell_0,\ell_1,\ell_2,\ell_3$ of $M_0,M_1,M_2,M_3$ respectively,
\[
\ms(\ell_0 \vee \ell_1 \vee \ell_2\vee \ell_3) \geq
\min\{\ms(\ell_0 \vee \ell_2)+m_1,\ms(\ell_1 \vee \ell_3)+m_2,\ms(\ell_0 \vee \ell_3)+m_1+m_2\}.
\]
\end{lem}
\begin{proof}
Let $G=M_0 \cup M_1 \cup M_2 \cup M_3$ and $\ell = \ell_0 \vee \ell_1 \vee \ell_2\vee \ell_3$. Let $e$ and $e'$ be a pair of adjacent edges in $G$ with $\ell(e)<\ell(e')$. Because $e$ and $e'$ are adjacent in $G$, we must have $e \in E(M_i)$ and $e' \in E(M_j)$ for some $(i,j) \in \{(0,2),(1,3),(0,3)\}$. Then $d_{\ell_i \vee \ell_j}(e,e') \geq  \ms(\ell_i \vee \ell_j)$ by definition. Also, $d_{\ell}(e,e') = d_{\ell_i \vee \ell_j}(e,e')+s$, where $s=m_1$ if $(i,j)=(0,2)$, $s=m_2$ if $(i,j)=(1,3)$ and $s=m_1+m_2$ if $(i,j)=(0,3)$. The result follows.
\end{proof}

\begin{lem}\label{lem:ord given other ord bad case}
Let $X$ and $Y$ be edge-disjoint matchings and $\ell_{Y}$ be a fixed ordering of $Y$.
Then there is an ordering $\ell_{X}$ of $X$ such that $\ms( \ell_{X} \vee \ell_{Y}) \geq \frac{1}{2}|E(X)|$.
\end{lem}
\begin{proof}
Let $x=|E(X)|$ and $y=|E(Y)|$. For each edge $e \in E(X)$, let $\alpha(e)$ be the smallest label assigned by $\ell_Y$ to an edge adjacent to $e$ if such a label exists, and $\alpha(e)=\infty$ otherwise. Let $\ell_X$ be an ordering $(e_0,\ldots,e_{x-1})$ of $X$ such that $\alpha(e_0) \leq \cdots \leq \alpha(e_{x-1})$. Let $\ell = \ell_{X} \vee \ell_Y$ and
$e$ and $e'$ be adjacent edges in $G$ such that $\ell(e)<\ell(e')$ and $d_{\ell}(e,e')=\ms(\ell)$. Then $e=e_i$ for some $i \in \Z_x$ such that $\alpha(e_i)<\infty$ and $\alpha(e_i)=\ell_Y(e')$. By our definition of $\ell_X$, any edge of $X$ that is not adjacent to an edge of $Y$ occurs after $e_i$ in $\ell_X$ and hence each edge in $\{e_0,\ldots,e_{i-1}\}$ is adjacent to at least one edge of $Y$. Thus, because at most two edges of $X$ are adjacent to each edge of $Y$, we have that $\alpha(e_i) \geq \lfloor\frac{i}{2}\rfloor$ and hence that $d_{\ell}(e,e') =x -i+\alpha(e_i)\geq x-i+\lfloor\frac{i}{2}\rfloor$. So, because $i \leq x-1$, we have $d_{\ell}(e,e') \geq 1+\lfloor\frac{x-1}{2}\rfloor = \lceil\frac{x}{2}\rceil$ and the result follows.
\end{proof}

\begin{lem}\label{lem:ord given other ord}
Let $X$ and $Y$ be edge-disjoint matchings of sizes $x$ and $y$ respectively.
Suppose that $Y$ has $y_1$ edges that are adjacent to one edge in $X$  and
$y_2$ edges adjacent to two edges in $X$.
Let $\ell_{Y}$ be an ordering of $Y$ in which the $y_2$ edges adjacent to two edges in $X$
are the last to occur. Then there is an ordering $\ell_{X}$ of $X$ such that
\[\ms( \ell_{X}\vee\ell_{Y}) \geq \min\{x,x+y-y_1-2y_2\}.\]
\end{lem}
\begin{proof}
For $i \in \{1,2\}$, let $Y_i$ be the set of edges of $Y$ that are adjacent to exactly $i$ edges in $X$. For each edge $e \in E(X)$, let $\alpha(e)$ be the smallest label assigned by $\ell_Y$ to an edge adjacent to $e$ if such a label exists, and $\alpha(e)=\infty$ otherwise. Let $\ell_X$ be an ordering $(e_0,\ldots,e_{x-1})$ of $X$ such that $\alpha(e_0) \leq \cdots \leq \alpha(e_{x-1})$.

Let $\ell = \ell_X \vee \ell_Y$ and
$e$ and $e'$ be adjacent edges in $G$ such that $\ell(e)<\ell(e')$ and $d_{\ell}(e,e')=\ms( \ell)$. Then $e=e_i$ for some $i \in \Z_x$ such that $\alpha(e_i)<\infty$ and $\alpha(e_i)=\ell_Y(e')$. So $d_{\ell}(e,e')=x-i+\alpha(e_i)$. By our definition of $\ell_X$, any edge of $X$ that is not adjacent to an edge of $Y$ occurs after $e_i$ in $\ell_X$ and hence each edge in $\{e_0,\ldots,e_{i-1}\}$ is adjacent to at least one edge of $Y$. We consider two cases.

Suppose that $e' \in Y_1$. Then each edge in $\{e_0,\ldots,e_{i-1}\}$ is adjacent to at least one edge of $Y_1$ (recall the edges in $Y_2$ occur last in $\ell_Y$) and hence $\alpha(e_i) \geq i$. It follows that $d_{\ell}(e,e')\geq x$ and the result is established.

Suppose instead that $e' \in Y_2$. Let $j$ be the smallest element of $\Z_x$ such that $e_j$ is not adjacent to an edge in $Y_1$ and note that $j \leq y_1$ and that $\alpha(e_j) \geq y-y_2$ because the edges of $Y_2$ occur last in $\ell_Y$. So, because at most two edges of $X$ are adjacent to each edge of $Y_2$, we have that $\alpha(e_i) \geq \alpha(e_j)+\lfloor\frac{i-j}{2}\rfloor \geq y-y_2+\lfloor\frac{i-j}{2}\rfloor$. Thus,
\[d_{\ell}(e,e') = x-i+\alpha(e_i) \geq x+y-y_2-\lceil\tfrac{i+j}{2}\rceil \,.\]
Now, we saw that $j \leq y_1$ and we must have $i \leq y_1+2y_2-1$ for otherwise $\alpha(e_i)=\infty$. Thus, $\lceil\frac{i+j}{2}\rceil \leq y_1+y_2$ and hence $d_{\ell}(e,e') \geq x+y-y_1-2y_2$, and again the result is established.
\end{proof}

\section{2-regular graphs}\label{2reg}
In this section we will prove Theorem~\ref{thm:2 reg}.
We will require the result of Brualdi et al.~\cite{MR2961987} on $\cms(C_n)$
that was mentioned in the introduction.
\begin{thm}[Brualdi et al.~\cite{MR2961987}]\label{thm: cycle seq}
For all $n \geq 3$, $\cms(C_n) =\left\lfloor \frac{n-1}{2} \right\rfloor$.
\end{thm}

We first prove a useful result that gives an ordering of a particular type for a class 1 graph that is either a single cycle or a union of vertex-disjoint paths.

\begin{lem}\label{lem:good ord 2-colourable}
Let $H_0$ and $H_1$ be edge-disjoint matchings such that $|E(H_0)|=|E(H_1)|=t$ for some integer $t \geq 2$ and $H_0 \cup H_1$ is either a single cycle or a union of vertex-disjoint paths. There exist orderings $\ell_0$ and $\ell_1$ of $H_0$ and $H_1$ respectively such that $\cms(\ell_0 \vee \ell_1) \geq t-1$.
\end{lem}
\begin{proof}
We proceed by induction on $|V(H_0 \cup H_1)|$. Let $H=H_0 \cup H_1$.
If $|V(H)|=2t$ then $H$ is a cycle of length $2t$. We may assume its vertex set is $\mathbb{Z}_{2t}$ and its edge set is
$\{e_{i}: i \in \mathbb{Z}_{2t} \}$, where $e_{i} = \{ i,i+1\}$ and
$H_j =\{e_{i}\;:\; i \in \mathbb{Z}_{2t} \,\text{ and }\, i\equiv j \pmod{2}\}$ for $j \in \Z_2$. Let
\[
\ell_j(e_{i}) = i\quad \text{and} \quad \ell_{j+1}(e_{2t-1-i})=i
\]
for all $ i \in \{0,\ldots, t-1\}$ and $j \in \Z_2$ such that  $j \equiv i \pmod{2}$. Note that $\ell_j$ is an ordering of $H_j$ for each $j \in \Z_2$. Let $\ell=\ell_0 \vee \ell_1$ and for an edge $e \in E(H)$, let $\ell^*(e)=\ell_j(e)$ where $j$ is the element of $\Z_2$ such that $e \in E(H_j)$. Let $\{e_{h-1},e_h\}$, where $h \in \mathbb{Z}_{2t}$, be an arbitrary pair of adjacent edges of $H$ and note that one of these edges is from $H_0$ and the other is from $H_1$. If $h \in \{1,\ldots,t-1\}$, then $d_\ell\{e_{h-1},e_h\}=t-1$ because $\ell^*(e_{h-1})=\ell^*(e_{h})-1$. Similarly, if $h \in \{t+1,\ldots,2t-1\}$, then $d_\ell\{e_{h-1},e_h\}=t-1$ because $\ell^*(e_{h-1})=\ell^*(e_{h})+1$. Finally, if $h \in \{0,t\}$, then $d_\ell\{e_{h-1},e_h\}=t$ because $\ell^*(e_{h-1})=\ell^*(e_{h})$. Thus it follows that $\cms(\ell_0 \vee \ell_1) = t-1$ and we have proved the result in the case where $|V(H)|=2t$.

Now suppose that $|V(H)|>2t$. Then $H$ is a union of $k$ disjoint paths for some $k \geq 1$. There are edges $yy' \in E(H_0)$ and $zz' \in E(H_1)$ such that $y \notin V(H_1)$, $z \notin V(H_0)$ and, if $k \geq 2$, then $y$ and $z$ are in different paths. Let $H'_0$ and $H'_1$ be the matchings obtained from $H_0$ and $H_1$ by merging the vertices $y$ and $z$ into a new vertex $x$. Then $H'_0 \cup H'_1$ is either a single cycle or a union of paths, and $|V(H'_0 \cup H'_1)|=|V(H_0 \cup H_1)|-1$. So, by induction, there are orderings $\ell'_0$ and $\ell'_1$ of $H'_0$ and $H'_1$, respectively such that $\cms(\ell'_0 \vee \ell'_1) \geq t-1$. Then $\cms(\ell_0 \vee \ell_1) \geq t-1$ where $\ell_0$ and $\ell_1$ are the orderings of $H_0$ and $H_1$ obtained from $\ell'_0$ and $\ell'_1$ by replacing $xy'$ with $yy'$ in $\ell_0$ and $xz'$ with $zz'$ in $\ell_1$. So the result follows by induction.
\end{proof}

Our next lemma implies that $\cms(G) \geq \frac{n-2}{2}$ for each 2-regular class 1 graph $G$ of order $n$, but also says more.

\begin{lem}\label{lem:2 edge colourable}
Let $H_0$ and $H_1$ be edge-disjoint matchings such that $|E(H_0)|=|E(H_1)|=t$ for some integer $t \geq 1$.
There are orderings $\ell_0$ and $\ell_1$ of $H_0$ and $H_1$, respectively,
such that $\cms(\ell_0 \vee \ell_1) \geq t-1$.
\end{lem}
\begin{proof}
We proceed by induction on $|E(H_0)|=|E(H_1)|$. Let $H=H_0 \cup H_1$. Clearly $H$ must have a subgraph $H^\dag$ such that $H^\dag$ is either
\begin{itemize}
    \item
a component of $H$ that has an even number of edges (and is either a path or a cycle); or
    \item
a union of two components of $H$, each of which is a path of odd length, with the property that $|E(H_0) \cap E(H^\dag)|=|E(H_1) \cap E(H^\dag)|$.
\end{itemize}
Let $|E(H^\dag)|=2s$, noting that $|E(H^\dag)|$ is even, and let $H^\dag_i$ be the matching of size $s$ with edge set $E(H_i) \cap E(H^\dag)$ for each $i \in \Z_2$. By Lemma~\ref{lem:good ord 2-colourable} there are orderings $\ell^\dag_0$ and $\ell^\dag_1$ of $H^\dag_0$ and $H^\dag_1$, respectively, such that $\cms(\ell^\dag_0 \vee \ell^\dag_1) \geq s-1$.

If $H^\dag=H$, then the result follows by taking $\ell_i=\ell^\dag_i$ for $i \in \Z_2$, so we may assume that $H^\dag\neq H$. For $i \in \Z_2$, let $H^\ddag_i$ be the matching of size $t-s$ with edge set $E(H_i) \setminus E(H^\dag_i)$. By our inductive hypothesis, there are orderings  $\ell^\ddag_0$ and $\ell^\ddag_1$ of $H^\ddag_0$ and $H^\ddag_1$ such that $\cms(\ell^\ddag_0 \vee \ell^\ddag_1) \geq t-s-1$. Let $\ell_i=\ell^\dag_i \vee \ell^\ddag_i$ for $i \in \Z_2$.

Any pair $\{e,e'\}$ of adjacent edges in $H$ such that $e \in E(H^\ddag_0)$ and $e' \in E(H^\ddag_1)$ are at distance at least $t-s -1$ in $\ell^\ddag_0 \vee \ell^\ddag_1$ because $\cms(\ell^\ddag_0 \vee \ell^\ddag_1) \geq t-s-1$, and hence are at distance at least $t-s-1+s=t-1$ in $\ell_0 \vee \ell_1$. Likewise, any pair $\{e,e'\}$ of adjacent edges in $H$ such that $e \in E(H^\dag_0)$ and $e' \in E(H^\dag_1)$ are at distance at least $s-1$ in $\ell^\dag_0 \vee \ell^\dag_1$, and hence are at distance at least $s-1+t-s=t-1$ in $\ell_0 \vee \ell_1$. Because $H^\dag$ and $H^\ddag$ are vertex-disjoint, these arguments cover all pairs of adjacent edges in $H$ and so $\cms(\ell_0 \vee \ell_1) \geq t-1$.
\end{proof}

\begin{lem}\label{lem:class 1 2 regular cms}
For each even $n \geq 4$, we have $\cms_1(n,2)=\frac{n-2}{2}$.
\end{lem}

\begin{proof}
By Theorem~\ref{thm: cycle seq}, for each even $n \geq 4$, we have $\cms(H) = \frac{n-2}{2}$ if $H$ is an $n$-cycle and hence $\cms(n,2) \leq \frac{n-2}{2}$. Also, Lemma~\ref{lem:2 edge colourable} implies that $\cms(H) \geq \frac{n-2}{2}$ for each $2$-regular class 1 graph $H$ of even order $n \geq 4$, because any such graph $H$ is the union of two edge-disjoint matchings each of size $\frac{n}{2}$.
\end{proof}

The \emph{matching number}  of a graph $G$ is the maximum size of a matching in $G$. If $\ell$ is an ordering of a graph $H$ and $G$ is a subgraph of $H$ then the \emph{subordering of $\ell$ induced by $G$} is the unique ordering $\ell_G$ of $G$ such that, for all $e, e' \in E(G)$, $\ell_G(e)<\ell_G(e')$ if and only if  $\ell(e)<\ell(e')$. An ordering $\ell^*$ is a \emph{subordering} of an ordering $\ell$ of a graph $H$ if $\ell^*$ is induced by $G$ for some subgraph $G$ of $H$. Our next lemma provides upper bounds on the cyclic matching sequenceability of a graph based on the properties of one of its subgraphs. We only need the simpler first part in this section, but the more involved second part is required in Section~\ref{genBounds}.

\begin{lem}\label{lem:gen bound given subG}
Let $H$ be a graph and $G$ be a subgraph of $H$ with matching number $\nu$. Then
\begin{itemize}
    \item[(i)]
$\cms(H) \leq \frac{\nu|E(H)|}{|E(G)|}$
    \item[(ii)]
$\cms(H) \leq \frac{|E(H)|}{\left\lfloor \frac{1}{\nu}(|E(G)|-\cms(G)) \right\rfloor+1}$.
\end{itemize}
\end{lem}
\begin{proof}
We first prove (i). Because any fractional edge colouring of $H$ can be restricted in the natural fashion to a fractional edge colouring of $G$ with equal or lesser weight, we have $c_f(H) \geq c_f(G)$. Furthermore, it is clear from the definition of fractional chromatic index that $c_f(G) \geq \frac{1}{\nu}|E(G)|$. Thus $c_f(H) \geq \frac{1}{\nu}|E(G)|$ and (i) follows from Lemma~\ref{lem:fracChromatic}.

We now prove (ii). Let $\ell$ be an ordering of $H$ and let $h=\lfloor \frac{1}{\nu}(|E(G)|-\cms(G))\rfloor+1$. We will find a subordering $\ell''=(e_0,\ldots,e_{h-1})$ of $\ell$ so that, for each $i \in \Z_{h}$, the sequence of cyclically consecutive edges in $\ell$ that begins with $e_i$ and ends with $e_{i+1}$ contains a pair of adjacent edges. This will suffice to complete the proof of (ii) because $d_\ell(e_j,e_{j+1}) \leq \frac{1}{h}|E(H)|$ for some $j \in \Z_h$ since $\sum_{i \in \Z_h}d_\ell(e_i,e_{i+1})=|E(H)|$.

Let $\ell'$ be the subordering of $\ell$ induced by $G$.
There must be two adjacent edges $e_0$ and $e_1$ of $G$ at distance at most $\cms(G)$ in $\ell'$ by the definition of $\cms(G)$.
Because we are considering $\ell$ and $\ell'$ cyclically, we can assume without loss of generality that $\ell'(e_0)=0$ and $\ell'(e_1)=a$ for some $a \leq \cms(G)$.
Now define $(e_0,\ldots,e_{h-1})$ by, for each $i \in \{2,\ldots,h-1\}$, letting $e_{i}$ be the first edge after $e_{i-1}$ in $\ell'$ such that the sequence of consecutive edges in $\ell'$ that begins with $e_{i-1}$ and ends with $e_{i}$ contains a pair of adjacent edges. We claim that $(e_0,\ldots,e_{h-1})$ is a subordering of $\ell$ with the required properties.
To see this, first observe that $a+(h-1)\nu \leq |E(G)|$ by the definition of $h$.
Thus $(e_0,\ldots,e_{h-1})$ is indeed a subordering of $\ell$ because, for each $i \in \{1,\ldots,h-2\}$, we have that $\ell'(e_{i-1}) < \ell'(e_i) \leq a + (i-1)\nu$, by the definition of $\nu$.
In particular, $\ell'(e_{h-1}) \leq a+(h-2)\nu \leq |E(G)|-\nu$ and hence the sequence of cyclically consecutive edges in $\ell'$ that begins with $e_{h-1}$ and ends with $e_{0}$ contains a pair of adjacent edges. So $(e_0,\ldots,e_{h-1})$ is a subordering of $\ell$ with the required properties and (ii) is proved.
\end{proof}

Applying Lemma~\ref{lem:gen bound given subG}(i) to $2$-regular class 2 graphs we obtain the following.
\begin{cor}\label{cor:lower bound m-cycle}
Let $H$ be a $2$-regular class 2 graph, whose shortest odd cycle has length $m$. Then $\cms(H) \leq \frac{m-1}{2m}|E(H)|$.
\end{cor}
\begin{proof}
Let $G$ be a shortest odd length cycle in $H$. Then $|E(G)|=m$ and the matching number of $G$ is $\frac{m-1}{2}$, so the result follows by applying Lemma~\ref{lem:gen bound given subG}(i).
\end{proof}

To prove Theorem~\ref{thm:2 reg}, it remains to show that $\cms(H) \geq \lfloor\frac{n}{3}\rfloor$ for each $2$-regular graph $H$ of order $n$. In Lemma~\ref{lem:noSingle4Cycle} we establish a slightly stronger result for all $2$-regular graphs that do not contain exactly one $4$-cycle. For the remainder of this section it will be convenient to denote the number of edges in an ordering $\ell$ by $|\ell|$.

\begin{lem}\label{lem:noSingle4Cycle}
Let $G$ be a $2$-regular graph such that $G$ does not contain exactly one $4$-cycle. Then there is an ordering $\ell= \ell_0 \vee \ell_1 \vee \ell_2$ of $G$ such that
\begin{itemize}
    \item[(i)]
the edges of $\ell_i$ form a matching of size $m_i$ for each $i \in \mathbb{Z}_3$, where $m_0,m_1,m_2$ are the unique non-negative integers such that $\lceil\frac{1}{3}|\ell|\rceil \geq m_0 \geq m_1 \geq m_2 \geq \lfloor\frac{1}{3}|\ell|\rfloor$ and $m_0+m_1+m_2=|\ell|$; and
    \item[(ii)]
$d_{\ell}(e,e') \geq |\ell_j|$ for any $j \in \mathbb{Z}_3$ and pair $(e,e')$ of adjacent edges in $G$ such that $e$ is in $\ell_j$ and $e'$ is in $\ell_{j+1}$.
\end{itemize}
\end{lem}

\begin{proof}
We proceed by induction on the number of components of the graph. Let $H$ be a $2$-regular graph such that $H$ does not contain exactly one $4$-cycle. Let $m=|E(H)|$ and let $m_0,m_1,m_2$ be the unique non-negative integers such that $\lceil\frac{m}{3}\rceil \geq m_0 \geq m_1 \geq m_2 \geq \lfloor\frac{m}{3}\rfloor$ and $m_0+m_1+m_2=m$. If $H$ is connected or if $H$ contains no odd cycles, then by Theorem~\ref{thm: cycle seq} or Lemma~\ref{lem:good ord 2-colourable}, there is an ordering $\ell$ of $H$ such that $\cms(\ell) = \lfloor\frac{m-1}{2}\rfloor$. Choose $\ell_0, \ell_1, \ell_2$ arbitrarily so that $\ell= \ell_0 \vee \ell_1 \vee \ell_2$ and $|\ell_i|=m_i$ for each $i \in \mathbb{Z}_3$.
Because $m = 3$ or $m \geq 5$ we have $\cms(\ell) \geq \lfloor\frac{m-1}{2} \rfloor \geq \lceil\frac{m}{3}\rceil$ and it follows that the edges of $\ell_i$ form a matching for each $i \in \mathbb{Z}_3$ and that $\ell$ obeys (i) and (ii). So we may suppose that $H$ has $t \geq 2$ components at least one of which is an odd length cycle, and that the lemma holds for $2$-regular graphs with fewer than $t$ components. Let $t^*$ be the number of 4-cycles in $H$, noting that $t^* \neq 1$.

Our strategy will be as follows. We will first choose nonempty subgraphs $H'$ and $H''$ of $H$ such that $H$ is the vertex-disjoint union of $H'$ and $H''$. Let $m'=|E(H')|$,  $m''=|E(H'')|$ and $m'_0,m'_1,m'_2$ be the unique nonnegative integers such that $\lceil\frac{m'}{3}\rceil \geq m'_0 \geq m'_1 \geq m'_2 \geq \lfloor\frac{m'}{3}\rfloor$ and $m'_0+m'_1+m'_2=m'$. We will then find orderings $\ell' = \ell_1' \vee \ell_2' \vee \ell_3'$ and $\ell'' = \ell_1'' \vee \ell_2'' \vee \ell_3''$ of $H'$ and $H''$, respectively, such that $\ell'$ obeys (i) and (ii) and the edges of $\ell''_i$ form a matching of size $m_i-m'_i$ for each $i \in \mathbb{Z}_3$. Finally, we will establish that the ordering $\ell = \ell_1 \vee \ell_2 \vee \ell_3$ of $H$, where $\ell_i=\ell_i' \vee \ell_{i}''$ for $i \in \mathbb{Z}_3$, obeys (i) and (ii).

For the rest of the proof we take $(e,e')$ to be an arbitrary pair of edges that are adjacent in $H$ and $j$ to be an element of $\mathbb{Z}_3$ such that $e \in \ell_j$ and $e' \in \ell_{j+1}$. Because $H$ is a vertex-disjoint union of $H'$ and $H''$, we have that either $e,e' \in E(H')$ or $e,e' \in  E(H'')$. By our construction of $\ell$,
\begin{numcases}{d_{\ell}(e,e') \geq}
   d_{\ell'}(e,e')+m_j-m'_j & \hbox{if $e,e' \in E(H')$} \label{eqn:2-reg class 2 ell'}
   \\
   d_{\ell''}(e,e')+m'_{j+1} & \hbox{if $e,e' \in E(H'')$} \label{eqn:2-reg class 2 ell''}
\end{numcases}
Because $\ell'$ will obey (ii), we will have for each $j \in \Z_3$ that $d_{\ell'}(e,e') \geq m'_j$ and hence $d_{\ell}(e,e') \geq  m_j$ by \eqref{eqn:2-reg class 2 ell'} if $e,e' \in E(H')$. Thus, when checking that $\ell$ satisfies (ii), it will suffice to only consider the case $e,e' \in  E(H'')$. So we henceforth assume that $e,e' \in  E(H'')$.

We now describe how we choose $H'$ and $H''$ and how to find orderings $\ell'$ and $\ell''$ with the appropriate properties.
\begin{itemize}
    \item
If $H$ has a cycle of length congruent to 0 modulo 3, choose $H'$ and $H''$ such that $H'$ is this cycle and $H$ is the vertex-disjoint union of $H'$ and $H''$. Note that $m'_i=\frac{1}{3}m'$ for each $i \in \Z_3$. Then $H'$ contains no $4$-cycle and $H''$ contains $t^*$ $4$-cycles (recall $t^*\neq 1$).  By induction, take orderings $\ell' = \ell_1' \vee \ell_2' \vee \ell_3'$ of $H'$ and $\ell'' = \ell_1'' \vee \ell_2'' \vee \ell_3''$ of $H''$ that obey (i) and (ii). Because $\ell'$ obeys (i), $\ell'_i$ is an ordering of a matching of size $\frac{1}{3}m'$ for each $i \in \Z_3$. Because $\ell''$ obeys (i), $\ell''_i$ is an ordering of a matching of size $m_i-\frac{1}{3}m'=m_i-m'_i$ for each $i \in \Z_3$ and it can be seen that $\ell$ also obeys (i). Because $\ell''$ obeys (ii), we have $d_{\ell''}(e,e') \geq |\ell''_j|=m_j-\frac{1}{3}m'$ and it can be seen by \eqref{eqn:2-reg class 2 ell''} that $d_{\ell}(e,e') \geq m_j$ and hence that $\ell$ obeys (ii).
    \item
Otherwise, choose $H'$ and $H''$ such that $H''$ is a single odd length cycle (recall $H$ contains at least one odd length cycle) and $H$ is the vertex-disjoint union of $H'$ and $H''$. Observe that $H''$ is not a $3$-cycle because we are not in the previous case, so $m'' \geq 5$ and $m''$ is odd. Then $H'$ contains $t^*$ $4$-cycles (recall $t^*\neq 1$) and, by induction, there is an ordering $\ell' = \ell_0' \vee \ell_1' \vee \ell_2'$ of $H'$ that obeys (i) and (ii). By Theorem~\ref{thm: cycle seq}, there is an ordering $\ell''$ of $H''$ such that $\cms(\ell'') = \lfloor\frac{1}{2}(m''-1) \rfloor$. Choose $\ell_0'', \ell_1'', \ell_2''$ arbitrarily so that $\ell''=\ell''_0 \vee \ell''_1 \vee \ell''_2$ and $|\ell''_i|=m_i-m'_i \leq \lceil\frac{1}{3}m''\rceil$ for each $i \in \mathbb{Z}_3$.
Because $m'' \geq 5$, we have $\cms(\ell'') \geq \lfloor\frac{1}{2}(m''-1) \rfloor \geq \lceil\frac{1}{3}m''\rceil$ and it follows that $\ell''_i$ is an ordering of a matching for each $i \in \mathbb{Z}_3$. Hence, because $\ell'$ obeys (i), we also have that $\ell$ obeys (i). By using $d_{\ell''}(e,e') \geq \lfloor\frac{1}{2}(m''-1) \rfloor$, $m''=m-m'$ and $m'' \geq 5$, we see that \eqref{eqn:2-reg class 2 ell''} implies that $d_{\ell}(e,e') \geq m_j$, and hence that $\ell$ obeys (ii) of the claim, provided that
\begin{equation}\label{eqn:2-reg class 2 mess}
\lfloor\tfrac{1}{2}(m-m'-1)\rfloor+m'_{j+1} \geq m_j
\end{equation}
holds for $3 \leq m' \leq m-5$. Let $\epsilon,\epsilon' \in \{-2,-1,0,1,2\}$ be such that $m_{j}=\frac{1}{3}(m+\epsilon)$ and $m'_{j+1}=\frac{1}{3}(m'+\epsilon')$. Because $m'_{j+1}$ and $m_j$ are integers, \eqref{eqn:2-reg class 2 mess} is equivalent to $m-m' > 2m_j-2m'_{j+1}$. Thus, substituting $m_{j}=\frac{1}{3}(m+\epsilon)$ and $m'_{j+1}=\frac{1}{3}(m'+\epsilon')$ and simplifying, \eqref{eqn:2-reg class 2 mess} is also equivalent to
\begin{equation}\label{eqn:2-reg class 2 epsilons}
m-m' > 2\epsilon-2\epsilon'.
\end{equation}
Clearly, \eqref{eqn:2-reg class 2 epsilons} holds when $m-m' \geq 9$, because $|\epsilon|,|\epsilon'| \leq 2$. This leaves the cases when $m-m'=5$ and $m-m'=7$, recalling that $m-m'=m''$ is odd and at least 5. If either $m \equiv 0 \mod{3}$ or $m' \equiv 0 \mod{3}$, then one of $\epsilon$ or $\epsilon'$ is 0 and hence \eqref{eqn:2-reg class 2 epsilons} holds. Thus, we can assume that $m \equiv 1 \mod{3}$ if $m-m' =5$ and $m \equiv 2 \mod{3}$ if $m-m' = 7$. In each of these cases it is now routine to check that \eqref{eqn:2-reg class 2 epsilons} holds, by considering subcases according to the value of $j$ (note that the values of $\epsilon$ and $\epsilon'$ are completely determined by the congruence class of $m$ modulo 3 and the value of $j$). \qedhere
\end{itemize}
\end{proof}

We are now ready to prove Theorem~\ref{thm:2 reg}.

\begin{proof}[\textbf{\textup{Proof of Theorem~\ref{thm:2 reg}.}}]
By Lemma~\ref{lem:class 1 2 regular cms}, we have $\cms_1(n,2)=\frac{n-2}{2}$ for each even $n \geq 4$. It remains to show that $\cms(n,2)=\lfloor\frac{n}{3}\rfloor$ for each integer $n \geq 6$. Let $n \geq 6$ be an integer. By Corollary~\ref{cor:lower bound m-cycle}, any $2$-regular graph $H$ of order $n$ containing a $3$-cycle has $\cms(H) \leq \lfloor\frac{n}{3}\rfloor$. Thus $\cms(n,2) \leq \lfloor\frac{n}{3}\rfloor$ and it suffices to show that each $2$-regular graph $H$ of order $n$ has $\cms(H) \geq \lfloor\frac{n}{3}\rfloor$.

Let $H$ be a $2$-regular graph of order $n$. If $H$ does not contain exactly one $4$-cycle, then the properties of the ordering $\ell$ of $H$ given by Lemma~\ref{lem:noSingle4Cycle} ensure that $\cms(\ell) \geq  \lfloor\frac{n}{3}\rfloor$. Thus, we may assume that $H$ contains exactly one $4$-cycle. Say $H$ is the vertex-disjoint union of $H'$ and $H''$, where $H''$ is the $4$-cycle. Let $e_0,e_1,e_2,e_3$ be the edges of $H''$ so that $e_0,e_2$ and $e_1,e_3$ each form a matching. Let $\ell' = \ell'_0 \vee \ell'_1 \vee \ell'_2$ be an ordering of $H'$ given by Lemma~\ref{lem:noSingle4Cycle}. Let $e^*$ be the last edge in $\ell'_1$ and let $\ell^*_1$ be the ordering obtained from $\ell'_1$ by removing $e^*$. Let
\[\ell=
\left\{
  \begin{array}{ll}
    \ell'_0 \vee (e_0) \vee \ell'_1 \vee (e_2) \vee \ell'_2 \vee (e_1,e_3) & \hbox{if $|E(H')| \equiv 0, 1 \mod{3}$} \\
    \ell'_0 \vee (e_0) \vee \ell^*_1 \vee (e_2,e^*) \vee \ell_2' \vee (e_1,e_3) & \hbox{if $|E(H')| \equiv 2 \mod{3}$.}
  \end{array}
\right.
\]
It is now routine to use the fact that $\ell'$ satisfies (i) and (ii) of Lemma~\ref{lem:noSingle4Cycle} to check that any pair of edges adjacent in $H'$ or in $H''$ are at distance at least $\lfloor\frac{n}{3}\rfloor$ in $\ell$ and hence that the lemma holds. Note that the fact that $\ell'$ satisfies (i) and (ii) of Lemma~\ref{lem:noSingle4Cycle} implies that $e^*$ is not adjacent in $H'$ to any edge in $\ell'_2$ when $|E(H')| \equiv 2 \mod{3}$.
\end{proof}

\section{Upper and lower bounds for general graphs}\label{genBounds}

In this section we find some upper and lower bounds on the cyclic matching sequenceability of general (possibly non-regular) graphs. In particular we will prove Theorem~\ref{thm:gen bounds}. We employ an easily proved result from \cite{McD}. We say that a matching decomposition of a graph is \emph{equitable} if the sizes of any two of the matchings differ by at most 1.

\begin{lem}[\cite{McD}]\label{lem:colEq}
Let $G$ be a graph with chromatic index $c$. For any $t \geq c$, there is an equitable matching decomposition of $G$ with $t$ matchings.
\end{lem}

Our next result establishes the lower bound in Theorem~\ref{thm:gen bounds}.

\begin{lem}\label{lem:gen lower bound}
For any graph $H$ with chromatic index $c$,
$\cms(H) \geq \lfloor\frac{1}{2c}|E(H)| \rfloor-1$.
\end{lem}
\begin{proof}
Let $H$ be a graph. Let $c$ be the chromatic index of $H$, $m=|E(H)|$ and $t = \lfloor\frac{m}{2c}\rfloor$. When $c=1$, $H$ is a matching and the result is trivial, so we can assume that $c\geq 2$. By Lemma~\ref{lem:colEq}, there is a matching decomposition $\{H_0,\ldots,H_{c-1}\}$ of $H$ such that $|E(H_i)| \geq \lfloor\frac{m}{c}\rfloor \geq 2t$ for each $i \in \mathbb{Z}_{c}$. For each $i \in \mathbb{Z}_{c}$, let $H_i'$ and $H_i''$ be vertex-disjoint subgraphs of $H_i$, each with $t$ edges, and note that by Lemma~\ref{lem:2 edge colourable} there are orderings $\ell_i''$ of $H_{i}''$ and $\ell_{i+1}'$
of $H_{i+1}'$ such that $\ms(\ell_{i}''\vee \ell_{i+1}') \geq t - 1$. For each $i \in \mathbb{Z}_{c}$, let $\ell_i=\ell'_i \vee \ell^*_i \vee \ell''_i$ be an ordering of $H_i$, where $\ell^*_i$ is an arbitrary (possibly empty) ordering of the edges in $E(H_i) \setminus E(H_i' \cup H_i'')$. Let $\ell=\ell_1 \vee \cdots \vee \ell_{c}$. We complete the proof by showing that $\cms(\ell) \geq t-1$.

Let $e,e'$ be adjacent edges in $H$. Obviously $d_{\ell}\{e,e'\} \geq t-1$ if it is not the case that both $e$ and $e'$ are in $H''_j \cup H'_{j+1}$ for some $j \in \Z_c$. But if both $e$ and $e'$ are in $H''_j \cup H'_{j+1}$, then $d_{\ell}\{e,e'\} \geq t-1$ because $\ms(\ell_{j}''\vee \ell_{j+1}') \geq t - 1$.
\end{proof}

For each $k \geq 2$, we define a graph $B_k$ with maximum degree $k$. If $k$ is even, let $B_k$ be a complete graph on $k+1$ vertices and, if $k$ is odd, let $B_k$ be the graph on $k+2$ vertices whose complement is the vertex-disjoint union of a path with $3$ vertices and a matching with $k-1$ vertices. In particular, the following depicts $B_3$.
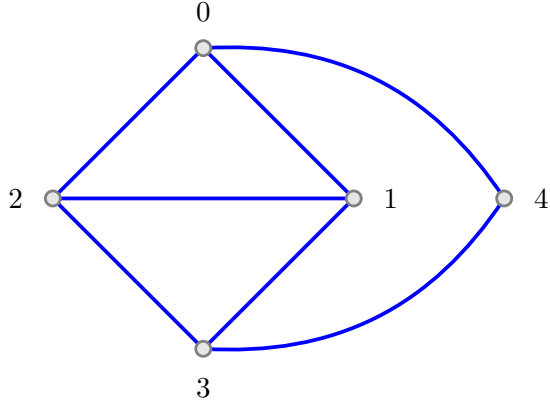
\begin{center}
\begin{tikzpicture}[thick,scale= .8]
  \pgfmathsetlengthmacro\scfac{2.5cm}
  \pgfmathsetmacro{\sepang}{120}
  \pgfmathtruncatemacro{\c}{3}
\draw[line width = \scfac*0.02, color = blue]{

(0+90:\scfac) -- (90+90:\scfac)
(0+90:\scfac) --  (270+90:\scfac)

(90+90:\scfac) -- (180+90:\scfac)
[bend left] (90+90:\scfac) -- node[pos = 0.5,above]{}(270+90:\scfac)

(180+90:\scfac) --  (270+90:\scfac)

[bend left] (0+90:\scfac) to ($(0:\scfac)+(270+90:\scfac)$)
[bend right] (180+90:\scfac) to ($(0:\scfac)+(270+90:\scfac)$)
};
\draw[line width = \scfac*0.015, color = black!50]{
(0+90:\scfac) node[anchor=north,yshift=\scfac*0.3,color = black]{0} node[circle, draw, fill=black!10,inner sep=\scfac*0.015, minimum width=\scfac*0.08] {}
(90+90:\scfac) node[anchor=east,xshift=-\scfac*0.1,color = black]{2}node[circle, draw, fill=black!10,inner sep=\scfac*0.015, minimum width=\scfac*0.08] {}
(180+90:\scfac) node[anchor=north,yshift=-\scfac*0.1,color = black]{3} node[circle, draw, fill=black!10,inner sep=\scfac*0.015, minimum width=\scfac*0.08] {}
(270+90:\scfac) node[circle, draw, fill=black!10,inner sep=\scfac*0.015, minimum width=\scfac*0.08] {}
(270+90:\scfac) node[anchor=west,xshift=\scfac*0.1,color = black]{1}
(0:\scfac)+(270+90:\scfac) node[anchor=west,xshift=\scfac*0.1,color = black]{4}node[circle, draw, fill=black!10,inner sep=\scfac*0.015, minimum width=\scfac*0.08] {}
};
\end{tikzpicture}
\end{center}
It is easy to check that when $k$ is odd $B_k$ has $\frac{1}{2}(k^2+2k-1)$ edges
and each of its vertices has degree $k$ except for one that has degree $k-1$.
Of course, when $k$ is even $B_k$ has $\frac{1}{2}k(k+1)$ edges and each of its vertices has degree $k$.

We will complete the proof of Theorem~\ref{thm:gen bounds} by showing that a graph $H$ containing $B_k$ as a subgraph has cyclic matching sequenceability at most $\frac{1}{k+1}{|E(H)|}$. We will make use of the following facts about $B_k$ for odd integers $k$.

\begin{lem}\label{lem:prop of bad sg}
For each odd $k \geq 3 $, $\cms(B_k)\leq \frac{k-1}{2}$ and $B_k$ has matching number $\frac{k+1}{2}$.
\end{lem}
\begin{proof}
Let $\nu$ be the matching number of $B_k$. It is easy to see that $\nu = \frac{k+1}{2}$, because $B_k$ has $k+2$ vertices, $k$ is odd, and it is easy to find a matching of size $\frac{k+1}{2}$ in $B_k$. So it remains to show that $\cms(B_k)\leq \frac{k-1}{2}$.
As the matching number of a graph $G$ is clearly an upper bound for $\cms(G)$, we only need to show that   $\cms(B_k)\neq \frac{k+1}{2}$.

Suppose for a contradiction that $\ell$ is an ordering of $B_k$ such that $\cms(\ell)= \frac{k+1}{2}$.
Let $v$ be the vertex of $B_k$ with degree $k-1$ and let $e_0 ,\ldots ,e_{k-2}$ be the edges of $B_k$ incident with $v$, where $\ell(e_i) < \ell (e_{i+1})$ for all $i \in \{0,\ldots ,k-3\}$.
Clearly $\sum_{i \in \Z_{k-1}}d_{\ell}(e_i,e_{i+1}) = |E(B_k)|$. So, for some $i \in \Z_{k-1}$ we have
\[
d_{\ell}(e_i,e_{i+1}) \geq \left\lceil\mfrac{|E(B_k)|}{k-1} \right\rceil =
 \left\lceil\mfrac{(k-1)(k+3)+2}{2(k-1)}\right\rceil = \mfrac{k+5}{2}.
\]
Therefore, in $\ell$, there are $\frac{k+3}{2}$ cyclically consecutive edges $e_{0}',\ldots, e_{(k+1)/2}'$
between $e_i$ and $e_{i+1}$, none of which are incident with $v$. As $\cms(\ell)= \frac{k+1}{2}$, we have that $e_{0}',\ldots, e'_{(k-1)/2}$ and $e_{1}',\ldots, e_{(k+1)/2}'$
must each form a matching of $B_k$. However, the two matchings so formed both have vertex set $V(H) \setminus \{v\}$, and they share $\frac{k-1}{2}$ edges. This is impossible, and we conclude that
$\cms(B_k) \leq  \frac{k-1}{2}$.
\end{proof}

\begin{thm}\label{thm:count exa}
Let $k \geq 2$ be an integer. Then $\cms(H) \leq \frac{1}{k+1}{|E(H)|}$ for any graph $H$ that has $B_k$ as a subgraph. Furthermore, for all integers $n \geq 3k+5$ such that $nk$ is even, there is
a $k$-regular graph on $n$ vertices with $B_k$ as a subgraph.
\end{thm}
\begin{proof}
Let $\nu$ be the matching number of $B_k$. When $k$ is even, $|E(B_k)|=\frac{1}{2}k(k+1)$, $\nu = \frac{k}{2}$ and the result follows by Lemma~\ref{lem:gen bound given subG}(i). When $k$ is odd, $|E(B_k)|=\frac{1}{2}(k^2+2k-1)$ and, by Lemma~\ref{lem:prop of bad sg}, $\cms(B_k) \leq \frac{k-1}{2}$ and $\nu = \frac{k+1}{2}$. So $|E(B_k)|-\cms(B_k) \geq \frac{1}{2}k(k+1)$ and the result can be seen to follow from  Lemma~\ref{lem:gen bound given subG}(ii).

Finally we show a $k$-regular graph on $n$ vertices with $B_k$ as a subgraph exists for any $n \geq 3k+4$ such that $nk$ is even.
If $k$ is even, then for any $n \geq 2k+2$, a graph that is the vertex-disjoint union of $K_{k+1}$ and a $k$-regular graph on $n-(k+1)$ vertices is a $k$-regular graph on $n$ vertices with $B_k$ as a subgraph. If $k$ is odd, then let $B'_k$ be the $k$-regular graph on $2k+4$ vertices that is formed by taking two vertex-disjoint copies of $B_k$ and adding an edge incident with the vertex of degree $k-1$ in each copy. For any $n \geq 3k+5$, a graph that is the vertex-disjoint union of $B'_{k}$ and a $k$-regular graph on $n-(2k+4)$ vertices is a $k$-regular graph on $n$ vertices with $B_k$ as a subgraph.
\end{proof}

Note that, for any $k\geq 2$, a $k$-regular graph containing $B_k$ is necessarily class 2. To see this let $x=|V(B_k)|$ and note that $x$ is odd and hence any matching in $B_k$ has size at most $\frac{x-1}{2}$. But, for both $k$ odd and $k$ even, $|E(B_k)| > \frac{k(x-1)}{2}$ and hence $B_k$ does not have a $k$-edge colouring.

\begin{proof}[\textbf{\textup{Proof of Theorem~\ref{thm:gen bounds}.}}]
The upper and lower bounds on $\cms(G)$ are established in Lemmas~\ref{lem:fracChromatic} and \ref{lem:gen lower bound}, respectively. Let $\Delta \geq 2$ and $n \geq \Delta+1$ be integers. If $n=\Delta+1$, then $K_{n}$ satisfies $\cms(K_n) \leq \frac{1}{\Delta+1}|E(K_{n})|$ by the result of \cite{MR2961987} mentioned in the introduction. If $n \geq \Delta+2$, then there is clearly a graph $G$ of order $n$ with maximum degree $\Delta$ that contains $B_\Delta$ as a subgraph. Then $\cms(G) \leq \frac{1}{\Delta+1}|E(G)|$ by Theorem~\ref{thm:count exa}.
\end{proof}

\section{Ordering with a given partition}\label{givenPartition}

Let $H$ be a $k$-regular graph with $n$ vertices and chromatic index $c$. In order to establish the lower bound in Theorem~\ref{thm:k reg}, we will construct an ordering $\ell$ of $H$ via a two stage process. In the first stage we will find a partition $\{X_i,Y_i,Z_i: i \in \mathbb{Z}_{c}\}$ of $E(H)$. In the second stage we will, for each $i \in \mathbb{Z}_{c}$, find orderings $\ell_{X_i}$, $\ell_{Y_i}$, $\ell_{Z_i}$ of $X_i$, $Y_i$, $Z_i$, respectively, then take
\[\ell= \ell_{Z_0} \vee \ell_{Y_0} \vee \ell_{X_0} \vee \ell_{Z_1} \vee \ell_{Y_1} \vee \ell_{X_1} \vee \cdots\cdots \vee \ell_{Z_{c-1}} \vee \ell_{Y_{c-1}} \vee \ell_{X_{c-1}}\]
and show that $\ell$ has the required matching sequenceability. In this section we detail how to construct the ordering $\ell$ given a partition of $E(H)$ with certain desirable properties. In Section~\ref{findPartition} we will establish that a partition with such properties does indeed exist.

Let $X_i \subseteq E(H_i)$ for all $i \in \mathbb{Z}_c$.
For each $i \in \Z_c$, we say
that a vertex $v$ in $V(H)$ is {\em $i$-covered for $\{X_0,\ldots , X_{c-1}\}$}
if $v$ is adjacent to an edge in $X_{i}$ and either there is an edge in $X_{i+1}$
that is also adjacent to $v$ or no edge in $H_{i+1}$ is adjacent to $v$.
For a graph $H$ and nonnegative integers $x$ and $w$, we say that a partition $\{X_i,Y_i,Z_i: i \in \mathbb{Z}_{c}\}$ of $E(H)$ is a \emph{$(x,w)$-partition} of $H$ if it obeys the following conditions.
\begin{itemize}[itemsep=0mm]
    \item[(P1)]
$w \leq \frac{3}{2}x$ and $x+2y \leq \lfloor \frac{1}{c}|E(H)| \rfloor$ where $y = \lceil 3x-\frac{3}{2}w\rceil$.
    \item[(P2)]
$\{H_i: i \in \Z_{c}\}$ is an equitable matching decomposition of $H$, where $H_i$ is subgraph of $H$ with edge set $X_i \cup Y_i \cup Z_i$ for each $i \in \Z_c$.
    \item[(P3)]
$|X_i| = x$ and $|Y_i| = y$ for all $i \in \mathbb{Z}_{c}$.
    \item[(P4)]
No edge in $X_i$ is adjacent to an edge in $Z_{i+1}$ for all $i \in \mathbb{Z}_{c}$.
    \item[(P5)]
For all $i \in \mathbb{Z}_{c}$, $|Y'_i| \leq \frac{y}{3}$ and each edge in $Y'_i$ is adjacent to at most one edge in  $Z_{i+1}$, where $Y_i'$ is the set of edges in $Y_i$ that are adjacent to exactly two edges in $X_{i-1}$.
    \item[(P6)]
For all $i \in \mathbb{Z}_{c}$, there are at least $w$ vertices of $H$ that are $i$-covered for $\{X_0,\ldots,X_{c-1}\}$.
\end{itemize}

We treat an $(x,w)$-partition as including a specification of which of its sets plays the role of $X_i$, $Y_i$ and $Z_i$ for each $i \in \mathbb{Z}_{c}$.  We will refer to these properties simply as (P1), (P2), \ldots, (P6) throughout the rest of the section and in the next section. When an $(x,w)$-partition is defined we will use $y$ and $Y'_i$ in the roles they play in (P1) and (P5) without explicitly defining them each time. Note that $|E(H_i)| \in \{\lfloor \frac{1}{c}|E(H)| \rfloor,\lceil \frac{1}{c}|E(H)| \rceil\}$ for each $i \in \Z_c$ because $\{H_i: i \in \Z_{c}\}$ is an equitable matching decomposition of $H$ by (P2). Thus, it follows from (P1) -- (P3) that $|Z_i|=|E(H_i)|-x-y\geq y$ for each $i \in \Z_c$.

Our goal for the rest of the section is to establish Proposition~\ref{prop:seq given some conds} which guarantees a lower bound on the cyclic matching sequenceability of a graph that admits an $(x,w)$-partition. Our next results, Lemmas~\ref{lem:mid tran to top gen k} and \ref{lem:bot tran to mid and bot gen k}, are used only in the proof of Proposition~\ref{prop:seq given some conds}. In Lemma~\ref{lem:mid tran to top gen k}, we define orderings of $Y_i$ and $Z_i$ for each $i \in \Z_c$ and then, based on these, in Lemma~\ref{lem:bot tran to mid and bot gen k} we determine orderings of $X_i$ for each $i \in \Z_c$.

\begin{lem}\label{lem:mid tran to top gen k}
Let $H$ be a graph and let $\{X_i,Y_i,Z_i: i \in \mathbb{Z}_{c}\}$ be a $(x,w)$-partition of $H$. For all $i\in \mathbb{Z}_{c}$, there are orderings $\ell_{Y_i}$ and $\ell_{Z_i}$ of $Y_i$ and $Z_i$ so that $\ms(\ell_{Y_i} \vee \ell_{Z_{i+1}}) \geq y-1$ and, in $\ell_{Y_i}$, the edges in $Y_i'$ are the last to occur.
\end{lem}

\begin{proof}
Fix an arbitrary $i\in \mathbb{Z}_{c}$. We will define orderings of $Y_i$ and $Z_{i+1}$.
As discussed above, (P1) -- (P3) imply that $|Z_i| \geq y$.
By (P5), $|Y_i'| \leq \frac{y}{3}$ and each edge of $Y_i'$ is adjacent to at most one edge in $Z_{i+1}$. Thus, we can choose a subset $Z_{i+1}'$ of $Z_{i+1}$ such that $|Z_{i+1}'|=|Y'_i|$ and $Z_{i+1}'$ includes every edge in $Z_{i+1}$ that is adjacent to an edge in $Y_i'$. Let $y'=|Y'_i|$. Choose an arbitrary ordering $\ell_{Y'_i}=(e_0,\ldots,e_{y'-1})$ of the edges in $Y'_i$. Because each edge of $Y_i'$ is adjacent to at most one edge in $Z_{i+1}$, we can now choose an ordering $\ell_{Z'_{i+1}}=(e^*_0,\ldots,e^*_{y'-1})$ of the edges in $Z_{i+1}'$ such that, for each $j \in \Z_{y'}$, either $e^*_j$ is not adjacent to any edge in $\ell_{Y'_i}$ or $e_j$ is the last of the (at most two) edges in $\ell_{Y'_i}$ adjacent to $e^*_j$. Clearly then,
\begin{equation}\label{eqn:topOrderings1}
\ms(\ell_{Y'_i} \vee \ell_{Z'_{i+1}}) \geq y'.
\end{equation}

Let $Y_i'' =Y_i\setminus Y_{i}' $, $Z_{i+1}'' = Z_{i+1}\setminus Z_{i+1}'$ and $y''=|Y_i''|$.
We have seen that (P1) -- (P3) imply $|Z_{i+1}|=|E(H_{i+1})|-x-y \geq y$ and thus, subtracting $y'$ from both sides, we have $|Z_{i+1}''| \geq y''$ and so we can find a subset $W$ of $Z_{i+1}''$
such that $|W|=y''$.
By Lemma~\ref{lem:2 edge colourable},  there are orderings $\ell_{Y''_i}$ and $\ell_{W}$ of the matchings formed by the edges of $Y_{i}''$ and the edges of $W$, respectively,
such that $\ms(\ell_{Y''_i} \vee \ell_{W}) \geq y''-1$.
Let $\ell_{Z''_{i+1}} = \ell_{W}\vee \ell_R$ where $\ell_R$ is an arbitrary ordering of the edges in $Z_{i+1}''\setminus W$.
Clearly then,
\begin{equation}\label{eqn:topOrderings2}
\ms(\ell_{Y''_i} \vee \ell_{Z''_{i+1}}) \geq y''-1.
\end{equation}
By the definition of $Z_{i+1}''$, no edge in it is adjacent to an edge in $Y'_i$. Thus, by Lemma~\ref{lem:adding 4 special orderings},
\[
\ms\left(\ell_{Y''_i}\vee \ell_{Y'_i} \vee \ell_{Z''_{i+1}} \vee \ell_{Z'_{i+1}}\right) \geq y'+y''-1 = y-1\,,
\]
where we have used the facts that $\ms(\ell_{Y''_i} \vee \ell_{Z''_{i+1}})+|Y'_i|\geq y''-1+y'$ by \eqref{eqn:topOrderings2}, that $\ms(\ell_{Y'_i} \vee \ell_{Z'_{i+1}})+|Z_{i+1}''| \geq y'+y''$ by \eqref{eqn:topOrderings1} and $|Y'_i|+|Z_{i+1}''| \geq y'+y''$.
So let $\ell_{Y_i} = \ell_{Y''_i} \vee \ell_{Y'_i}$ and $\ell_{Z_{i+1}} = \ell_{Z''_{i+1}}\vee \ell_{Z'_{i+1}}$, and note we have shown that these orderings satisfy $\ms(\ell_{Y_i} \vee \ell_{Z_{i+1}}) \geq y-1$. By applying this procedure for each $i \in \Z_{c}$, the lemma is proved.
\end{proof}

\begin{lem}\label{lem:bot tran to mid and bot gen k}
Let $H$ be a graph and let $\{X_i,Y_i,Z_i: i \in \mathbb{Z}_{c}\}$ be a $(x,w)$-partition of $H$. For all $i\in \mathbb{Z}_{c}$, let $\ell_{Y_i}$ and $\ell_{Z_i}$ be orderings of $Y_i$ and $Z_i$ that satisfy the conditions of Lemma~\ref{lem:mid tran to top gen k}. Then, for all $i \in \mathbb{Z}_{c}$, there is an ordering $\ell_{X_i}$ of $X_i$ such that $\ms(\ell_{X_i} \vee \ell_{Y_{i+1}}) \geq x$ and
$\ms(\ell_{X_i} \vee \ell_{X_{i+1}}) \geq x-y$.
\end{lem}
\begin{proof}
For each $i\in \mathbb{Z}_{c}$, we will find a subset $X'_i$ of $X_i$ such that $|X_i'| = \min\{x,\lfloor \frac{2y}{3} \rfloor\}$ and $X'_i$ includes every edge of $X_i$ that is adjacent to an edge in $Y'_{i+1}$, and then construct an ordering $\ell_{X'_i}$ of $X_i'$ for each $i\in \mathbb{Z}_{c}$. Once this is accomplished we will then find an ordering $\ell_{X''_i}$ of $X_i''=X_i \setminus X'_i$ for each $i\in \mathbb{Z}_{c}$, and show that the orderings $\ell_{X_i}=\ell_{X'_i}\vee \ell_{X''_i}$ satisfy the conditions of the lemma. Note that if $x \leq \lfloor \frac{2y}{3} \rfloor$, then the $X_i''$ will be empty and the ordering $\ell_{X''_i}$ will be trivial.

Let $j$ be an arbitrary element of $\mathbb{Z}_{c}$. Let $y_1$ be the number of edges of $Y_{j+1}$ adjacent to exactly one edge in $X_j$ and $y_2$ be the number of edges of $Y_{j+1}$ adjacent to exactly two edges in $X_j$, and note that $y_2=|Y_{j+1}'|$ by the definitions of $y_2$ and $|Y_{j+1}'|$. Consider the number of vertices incident with both an edge in $X_j$ and an edge in $Y_{j+1}$. Because $X_j$ is a matching, each edge in $Y_{j+1}$ is adjacent to at most two of its edges, and hence this number is $y_1+2y_2$. On the other hand, by (P6), this number is at most $2x-w$, where we note that $|X_j|=x$ and that the edges of $Y_{j+1} \cup X_{j+1}$ form a matching. It thus follows from  $y = \lceil 3x-\frac{3}{2}w\rceil$ that
\begin{equation}\label{eqn:y1y2Bound}
y_1+2y_2 \leq 2x-w \leq \mfrac{2y}{3}.
\end{equation}
Thus, we can choose a subset $X'_j$ of $X_j$ such that $|X_j'| = \min\{x,\lfloor \frac{2y}{3} \rfloor\}$ and $X'_j$ includes every edge of $X_j$ that is adjacent to an edge in $Y_{j+1}$ (if $x \leq \lfloor \frac{2y}{3} \rfloor$ then we choose $X'_j=X_j$). Further, because the last edges of $\ell_{Y_{j+1}}$ are those in $Y_{j+1}'$ and $y \geq y_1+2y_2$ by \eqref{eqn:y1y2Bound}, we can apply Lemma~\ref{lem:ord given other ord} to obtain an ordering $\ell_{X'_j}$ of $X_j'$ such that
\begin{equation}\label{eqn:bottomOrderings1}
\ms(\ell_{X'_j} \vee \ell_{Y_{j+1}}) \geq |X'_j|.
\end{equation}

Thus, for each $i\in \mathbb{Z}_{c}$, we can take such an ordering $\ell_{X'_i}$ of $X_i'$, let $\ell_{X_i}=\ell_{X'_i}$ if $x \leq \lfloor \frac{2y}{3} \rfloor$ and let $X''_i=X_i \setminus X'_i$ otherwise. If $x \leq \lfloor \frac{2y}{3} \rfloor$ then this completes the proof of the lemma, using \eqref{eqn:bottomOrderings1} and the fact that $x \leq y$. Thus, we may assume that $x > \lfloor \frac{2y}{3} \rfloor$. For each $i\in \mathbb{Z}_{c}$ by Lemma~\ref{lem:ord given other ord bad case} there is an ordering $\ell_{X''_i}$ of $X''_i$ such that
\begin{equation}\label{eqn:bottomOrderings2}
\ms\left(\ell_{X''_i} \vee \ell_{X'_{i+1}}\right) \geq \tfrac{1}{2}|X_i''| = \tfrac{1}{2}\left(x-\big\lfloor\tfrac{2y}{3}\big\rfloor\right) \geq x-y
\end{equation}
where the last inequality follows because $y \geq \frac{3}{4}x$ since $y = \lceil 3x-\frac{3}{2}w\rceil$ and $w \leq \frac{3}{2}x$ by (P1).

Again, let $j$ be an arbitrary element of $\mathbb{Z}_{c}$. As no edge in $X_j''$ is adjacent to an edge in $Y_{j+1}$, we have from \eqref{eqn:bottomOrderings1} that
\[\ms(\ell_{X'_j}\vee \ell_{X''_j} \vee \ell_{Y_{j+1}}) \geq |X'_j|+|X''_j| = x.\]
Obviously, $|X_j''| = x- \lfloor\frac{2y}{3}\rfloor \geq x-y$ and, because $y \geq \frac{3x}{4}$ by (P1), we have $|X_{j+1}'| = \lfloor\frac{2y}{3} \rfloor\geq x-y$. Therefore, by \eqref{eqn:bottomOrderings2},
\[\ms\left(\ell_{X'_j}\vee \ell_{X''_j} \vee \ell_{X'_{j+1}}\vee \ell_{X''_{j+1}}\right) \geq x-y.\]
Thus, the orderings $\ell_{X_i}=\ell_{X'_i}\vee \ell_{X''_i}$ for $i\in \mathbb{Z}_{c}$ satisfy the required properties.
\end{proof}

\begin{prop}\label{prop:seq given some conds}
If $H$ is a graph that has a $(x,w)$-partition for some nonnegative integers $x$ and $w$, then $\cms(H) \geq x+y-1$.
\end{prop}

\begin{proof}
Let $H$ be a graph and let $\{X_i,Y_i,Z_i: i \in \mathbb{Z}_{c}\}$ be a $(x,w)$-partition of $H$. By Lemmas~\ref{lem:mid tran to top gen k} and~\ref{lem:bot tran to mid and bot gen k} there are, for each $i \in \Z_{c}$, orderings $\ell_{X_i}, \ell_{Y_i}, \ell_{Z_i}$ of $X_i, Y_i, Z_i$, respectively, such that $\ms(\ell_{Y_i} \vee \ell_{Z_{i+1}}) \geq y-1$, $\ms(\ell_{X_i} \vee \ell_{Y_{i+1}}) \geq x$ and $\ms(\ell_{X_i} \vee \ell_{X_{i+1}}) \geq x-y$. Let $\ell_i = \ell_{Z_i} \vee \ell_{Y_i} \vee \ell_{X_i}$ for each $i \in \Z_{c}$. Now let $i$ be an arbitrary element of $\Z_{c}$. By Proposition~\ref{prop: Matching decomposition}, it suffices to show that $\ms(\ell_{i} \vee \ell_{i+1}) \geq x+y-1$.

We have $\ms(\ell_{Y_i} \vee \ell_{Z_{i+1}}) \geq y-1$. So, because $|Y_i| =y$ and $ Y_i \cup Z_i$ is a matching, we have
\begin{equation}\label{eqn:allOrderings1}
\ms(\ell_{Z_i} \vee \ell_{Y_i} \vee\ell_{Z_{i+1}} ) \geq y-1.
\end{equation}
We also have $\ms(\ell_{X_i} \vee \ell_{Y_{i+1}}) \geq x$ and $\ms(\ell_{X_i} \vee \ell_{X_{i+1}}) \geq x-y$.
Thus, because $\ms(\ell_{Y_{i+1}} \vee \ell_{X_{i+1}}) =y+x$, Lemma~\ref{lem:adding three orderings} implies that
\begin{equation}\label{eqn:allOrderings2}
\ms(\ell_{X_i} \vee \ell_{Y_{i+1}} \vee \ell_{X_{i+1}}) \geq \min\{x,x+y,y+(x-y) \} =x.
\end{equation}
By  (P4), the edges of $X_{i} \cup Z_{i+1}$ form a matching. Thus, applying Lemma~\ref{lem:adding 4 special orderings} with $M_1=Z_i \cup Y_i$, $M_2=X_{i}$, $M_3=Z_{i+1}$, $M_4=Y_{i+1} \cup X_{i+1}$, and using \eqref{eqn:allOrderings1} and \eqref{eqn:allOrderings2}, we have
\[
\ms(\ell_{i} \vee \ell_{i+1})
\geq
\min\{y-1+x,x+|Z_{i+1}|,\ms(\ell_{Z_i} \vee \ell_{Y_i} \vee \ell_{Y_{i+1}} \vee \ell_{X_{i+1}})+x+|Z_{i+1}| \} = x+y-1,
\]
where the last inequality holds because $|Z_{i+1}| \geq y$ which we have seen follows from (P1) -- (P3).
\end{proof}

\section{Finding a good partition}\label{findPartition}
In this section we prove Theorem~\ref{thm:k reg} by
establishing the existence of $(x,w)$-partitions in $k$-regular graphs with $k\geq 3$.
Let $H$ be a $k$-regular graph and $\{H_0,\ldots,H_{c-1}\}$
be an equitable matching decomposition of $H$.
Then, we call $\{X_0,\ldots, X_{c-1} \}$ an {\em $(x,w)$-semipartition with respect to $\{H_0,\ldots,H_{c-1}\}$}
if $X_i \subseteq E(H_i)$, $|X_i|=x$ for each $i \in \Z_c$ and $\{X_0,\ldots, X_{c-1} \}$ obeys (P6) for $w$. Our strategy is to first establish that it is possible to extend an $(x,w)$-semipartition
to an $(x,w)$-partition in Lemma~\ref{lem:can make T with S}, then to exhibit $(x,w)$-semipartitions
(using several different methods) in Lemmas~\ref{lem:gen blob construction}--\ref{lem:prob cons}.
In Lemmas~\ref{thm:main thm explicit con} and \ref{thm:prob con}, we then prove the lower bounds of Theorem~\ref{thm:k reg}, using Proposition~\ref{prop:seq given some conds}. Finally, we prove Theorem~\ref{thm:k reg}.

\begin{lem}\label{lem:can make T with S}
Let $k \geq 3$ be an integer, let $H$ be a $k$-regular graph with $n$ vertices, and let $\{H_0,\ldots,H_{c-1}\}$ be an equitable matching decomposition of $H$.
Let $\{X_0,\ldots, X_{c-1}\}$ be an $(x,w)$-semipartition.
If $x$ and $w$ satisfy \textup{(P1)}, then there exists
an $(x,w)$-partition of $H$.
\end{lem}
\begin{proof}
For $i \in \Z_c$, let $T_i$ be the set of edges in $E(H_i)\setminus X_i$ that are adjacent to exactly one edge in $X_{i-1}$ and let $T'_i$ be the edges in $E(H_i)\setminus X_i$ that are adjacent to exactly two edges in
$X_{i-1}$.
There are $2x$ vertices that are incident with an edge in $X_{i-1}$. Of these $2x$ vertices, $|T_i|$ are incident with an edge in $T_i$, $2|T'_i|$ are incident with an edge in $T'_i$, and by (P6) at least $w$
are incident with an edge in $X_{i}$ or have no edge of $H_i$ incident with them. Thus,
\begin{equation}\label{eqn:deg count}
w+|T_i|+2|T'_i| \leq 2x \,.
\end{equation}

For all $i \in\mathbb{Z}_{c}$, we construct sets $T''_{i+1}$ with the following properties.
\begin{itemize}
\item[(i)]
The set $T''_{i+1}$ is a subset of $E(H_{i+1})\setminus (X_{i+1} \cup T_{i+1}\cup T'_{i+1})$.
\item[(ii)]
For each edge in $T_i'$ that is adjacent to two edges in $E(H_{i+1})\setminus (X_{i+1} \cup T_{i+1}\cup T'_{i+1})$, at least one of these latter two edges is in $T''_{i+1}$.
\end{itemize}
Let $j$ be an arbitrary element of $\mathbb{Z}_{c}$. For each $e \in T_j'$ that is adjacent to two edges in $E(H_{j+1})\setminus (X_{j+1} \cup T_{j+1}\cup T'_{j+1})$, choose one of these adjacent edges, and let $T''_{j+1}$ be the set of all these chosen edges. Then clearly $T''_{j+1}$ has the desired properties and $|T''_{j+1}| \leq |T'_{j}|$. By \eqref{eqn:deg count}, $|T'_j| \leq x-\frac{w}{2}$. Thus, $|T''_{j+1}| \leq x-\frac{w}{2}$ and hence by (\ref{eqn:deg count})
\begin{equation}\label{eqn:num edges mid sec}
|T_{j}\cup T_{j}'\cup T_j''| = |T_j|+|T'_j|+|T''_j| \leq 2x-w+x-\tfrac{1}{2}w = 3x-\tfrac{3}{2}w \leq y.
\end{equation}
Now we let $Y_i \supseteq T_{i}\cup T_{i}'\cup T_i''$ be a $y$-subset of $E(H_i)\setminus X_i$  and
$Z_i = E(H_i)\setminus (X_i\cup Y_i)$ for all $i \in \mathbb{Z}_{c}$ (such a $Y_i$ exists because $x$ and $w$ obey (P1) and so we have $|E(H_i) \setminus X_i| \geq \lfloor \frac{1}{c}E(H)\rfloor-x \geq 2y$). We complete the proof by
showing that
$\{X_i,Y_i,Z_i: i \in \mathbb{Z}_{c}\}$ is an $(x,w)$-partition.
By our hypotheses, (P1) is satisfied, and (P2) and (P3) are immediate from the above construction. Because $Y_{j+1} \supseteq T_{j+1} \cup T_{j+1}'$, each edge of $E(H_{j+1})$ that is adjacent to an edge in $X_j$ is in $X_{j+1} \cup Y_{j+1}$. Thus, no edge of $X_j$ is adjacent to an edge in $Z_{j+1}$, as required for (P4). The set of edges in $Y_{j}$ that are adjacent to two edges in $X_{j-1}$ is $T_j'$ and it follows from (\ref{eqn:deg count}) and the definition of $y$
that $|T_j'| \leq x-\frac{w}{2} \leq \frac{y}{3}$. Also, by (ii) and because $Y_{j+1} \supseteq T_{j+1}''$, each edge of $T_j'$ is adjacent to at least one edge in $X_{j+1} \cup Y_{j+1}$ and hence is adjacent to  at most one edge in $Z_{j+1}$. Thus, (P5) holds. Because $\{X_0,\ldots, X_{c-1}\}$ is a $(x,w)$-semipartition, (P6) is satisfied.
\end{proof}

We now find $(x,w)$-semipartitions using two different approaches. The first is constructive and works better for small values of $k$. We detail it for class 1 graphs in Lemma~\ref{lem:gen blob construction} and for class 2 graphs in Lemma~\ref{lem:gen blob construction C2}. Our second approach is probabilistic and works better for large values of $k$. We detail it in Lemma~\ref{lem:prob cons}.

For the remainder of the section, it will be convenient to extend our existing notation slightly. Let $\{H_0,\ldots,H_{c-1}\}$ be an equitable matching decomposition of a graph $H$ and
let $X$ be a subset of $E(H)$. For $v \in V(H)$ and $i \in \Z_c$ we say that $v$ is {\em $i$-covered for $X$} if $v$ is $i$-covered for $\{X \cap E(H_0),\ldots,X \cap E(H_{c-1})\}$. That is, $v$ is $i$-covered for $X$ if $v$ is adjacent to an edge in $X \cap H_i$ and either there is an edge in $X \cap H_{i+1}$ that is also adjacent to $v$ or no edge in $H_{i+1}$ is adjacent to $v$. Also, for a graph $G$ and a subset $S$ of $V(G)$ we use $G[S]$ to denote the subgraph of $G$ induced by $S$.

\begin{lem}\label{lem:gen blob construction}
Let $k \geq 3$ be an integer. Let $H$ be a $k$-regular class $1$ graph with $n$ vertices, and let $\{H_0,\ldots,H_{k-1}\}$ be an equitable matching decomposition of $H$. Then for any $x \leq \frac{n}{2}$ there is an $(x,w)$-semipartition of $H$
with
$w = x+\lfloor\frac{x-1}{k-1}\rfloor$.
\end{lem}

\begin{proof}
Let $w=x+\lfloor\frac{x-1}{k-1}\rfloor$. Let $V_2$ be a set of any two adjacent vertices in $H$. We will iteratively define a sequence $V_2,\ldots,V_{w}$ of subsets of $V(H)$ such that $V_2 \subseteq \cdots \subseteq V_{w}$ and, for each $i \in \{2,\ldots,w\}$,
\begin{itemize}
    \item[(i)]
$|V_i|=i$;
     \item[(ii)]
at least $b$ of the graphs in $\{H_j[V_i]:j\in \Z_k\}$ have at least $a+1$ edges and the rest have at least $a$ edges, where $a$ and $b$ are the integers such that $i-1=ak+b$ and $b \in \{0,\ldots,k-1\}$.
\end{itemize}

Note that $V_2$ obeys (i) and (ii). Suppose inductively that for some $h \in \{2,\ldots,w-1\}$ we have a set $V_h$ obeying (i) and (ii). Let $a'$, $b'$, $a''$ and $b''$ be the integers such that $h-1=a'k+b'$, $h=a''k+b''$ and $b',b'' \in \{0,\ldots,k-1\}$. Notice that
\begin{equation}\label{eqn:divs and mods}
(a'',b'')=
\left\{
  \begin{array}{ll}
    (a'+1,0) & \hbox{if $h \equiv 0 \mod{k}$} \\
    (a',b'+1) & \hbox{otherwise.}
  \end{array}
\right.
\end{equation}
Let $j_0 \in \Z_k$ such that $|E(H_{j_0}[V_h])| \leq |E(H_j[V_h])|$ for each $j \in \Z_k$. If $|E(H_{j_0}[V_h])| \geq a'+1$, then $|E(H_{j}[V_h])| \geq a'+1$ for each $j \in \Z_k$ by the definition of $j_0$. In this case we take $V_{h+1}=V_h \cup \{u\}$ for any vertex $u \in V(H) \setminus V_h$ and note that $V_{h+1}$ obeys (i) and (ii) using \eqref{eqn:divs and mods}. If $|E(H_{j_0}[V_h])|=a'$ then, because $a'<\frac{h}{2}$, there is a vertex $u \in V(H) \setminus V_h$ such that the edge of $H_{j_0}$ incident with $u$ is also incident with a vertex in $V_h$. We take $V_{h+1}=V_h \cup \{u\}$. Then we have $|E(H_{j_0}[V_{h+1}])| = a'+1$. From this it can be checked, using \eqref{eqn:divs and mods}, that $V_{h+1}$ obeys (i) and (ii). So we have defined $V_2,\ldots,V_w$.

For each $j \in \Z_k$, let $X^*_j$ be the set of all edges of $H_j$ adjacent to at least one vertex in $V_w$ and observe that
\[|X^*_j|=w-|E(H_j[V_w])| \leq w-\left\lfloor\tfrac{w-1}{k}\right\rfloor\]
where the inequality follows because $V_w$ obeys (ii).
Now $w-\frac{w-1}{k} \leq x$ because $w \leq \frac{xk-1}{k-1}$ by definition,
and hence $w-\lfloor\frac{w-1}{k}\rfloor \leq x$ because $w$ and $x$ are integers.
Thus, for each $j \in \Z_k$, we can choose a subset $X_j$ of $E(H_j)$
such that $X^*_j \subseteq X_j$ and $|X_j|=x$.
Now, for each $j \in \Z_k$ and each $u \in V_w$,
there is an edge of $X_j$ and an edge of $X_{j+1}$ incident with $u$.
Therefore, $\{X_0 \ldots, X_{k-1}\}$ satisfies property (P6)
for $w$ and thus is an $(x,w)$-semipartition with respect to $\{H_0,\ldots,H_{k-1}\}$.
\end{proof}

\begin{lem}\label{lem:gen blob construction C2}
Let $k \geq 3 $ be an integer, let $H$ be a $k$-regular class 2 graph with $n \geq 6(k+1)$ vertices, and let $\{H_0,\ldots,H_k\}$ be an equitable matching decomposition of $H$. For any $x \leq \lfloor \frac{nk}{2(k+1)}\rfloor$ there is an $(x,w)$-semipartition of $H$, where $w =  x+\lfloor\frac{x-1}{k}\rfloor$.
\end{lem}

\begin{proof}
Throughout this proof, for any subset of $E(H)$ denoted $X(h)$ and any $j \in \Z_{k+1}$, we denote $X(h) \cap E(H_j)$ by $X_j(h)$. Let $w =  x+\lfloor\frac{x-1}{k}\rfloor$.  We claim there is a sequence $X(1),\ldots, X(w)$ of subsets of $E(H)$ such that $X(1) \subseteq \cdots \subseteq X(w)$ and, for all $i \in \{1,\ldots,w\}$,
\begin{itemize}
    \item[(i)]
$|X_j(i)| \in \{i -\lfloor\frac{i-1}{k+1}\rfloor-1,i -\lfloor\frac{i-1}{k+1}\rfloor\}$ for each $j \in \Z_{k+1}$;

    \item[(ii)]  $|\{j \in \mathbb{Z}_{k+1}: |X_j(i)| =i -\lfloor\frac{i-1}{k+1}\rfloor\}|  = k+1- i'$, where $i'$ is the least nonnegative integer congruent to $i-1$ modulo $k+1$;
     \item[(iii)] at least $i$ vertices are $j$-covered for $X(i)$ for each $j \in \Z_{k+1}$.
\end{itemize}

Suppose for the moment that this claim holds. For all $j \in \mathbb{Z}_{k+1}$ observe that
\[|X_j(w)|\leq w-\left\lfloor\tfrac{w-1}{k+1}\right\rfloor \leq  x\]
where the first inequality follows because $X(w)$ satisfies (i) and the second follows because $w \leq  x+\frac{x-1}{k}$ by the definition of $w$ and the fact that $w$ and $x$ are integers. Thus, for each $j \in \Z_k$, we can find a subset $X_j$ of $E(H_j)$ such that $X_j(w) \subseteq X_j$ and $|X_j|=x$. Then, because $X(w)$ satisfies (iii), we have that $\{X_0 \ldots, X_{k-1}\}$ satisfies property (P6) for $w$ and thus is an $(x,w)$-semipartition with respect to $\{H_0,\ldots,H_{k}\}$.

So it only remains to prove the claim. We do so by induction on $i$. Let $u \in V(H)$ be a vertex not incident in $H$ to an edge in $H_{0}$ and $v \in V(H)$  be a vertex not incident in $H$ to an edge in $H_{1}$ (such a vertex $v$ exists because $\{H_0,\ldots,H_k\}$ is equitable and $n \geq 6(k+1)$). Let $X(1)$ be the set containing each edge of $H$ incident with $u$ and the unique edge of $H_0$ incident with $v$. It is easy to check that $X(1)$ satisfies (i), (ii) and (iii).

Now suppose inductively that for some $h \in \{1,\ldots,w-1\}$ there is a set $X(h)$ that satisfies (i), (ii) and (iii).
We will show that there is a choice for $X(h+1)$ that satisfies (i), (ii) and (iii). Let $s$ be any element of $\mathbb{Z}_{k+1}$ such that $|X_s(h)| = h -\lfloor\frac{h-1}{k+1}\rfloor$ (note that at least one such exists by (ii)). We will construct $X(h+1)$ as $X(h) \cup \{e_j: j \in \Z_{k+1} \setminus \{s\}\}$ where $e_j$ is an edge in $E(H_j) \setminus X_j(h)$ for each $j \in \Z_{k+1} \setminus \{s\}$. It can be checked that this will ensure that $X(h+1)$ satisfies (i) and (ii) for $i=h+1$. (To see this, note that if $h \not\equiv 0 \mod{k+1}$ then $\lfloor\frac{h}{k+1}\rfloor=\lfloor\frac{h-1}{k+1}\rfloor$ and that if $h \equiv 0 \mod{k+1}$ then $\lfloor\frac{h}{k+1}\rfloor=\lfloor\frac{h-1}{k+1}\rfloor+1$ and $|X_j(h)| = h-1-\lfloor\frac{h-1}{k+1}\rfloor=h-\lfloor\frac{h}{k+1}\rfloor$ for all $j \in \Z_{k+1} \setminus \{s\}$ by (ii).) So our goal is to ensure that (iii) also holds.

We will first choose a subset $T$ of $\Z_{k+1} \setminus \{s\}$ and an $e_j$ for each $j \in T$ such that, for all $j \in T \cup \{s\}$, at least $h+1$ vertices are $j$-covered for $X'(h+1)$, where $X'(h+1)=X(h) \cup \{e_j: j \in T\}$. If more than $h$ vertices are $s$-covered for $X(h)$, then we can take $T=\emptyset$ and $X'(h+1)=X(h)$, so assume otherwise that precisely $h$ vertices are $s$-covered for $X(h)$. Thus, by our choice of $s$, there are $2(h -\lfloor\frac{h-1}{k+1}\rfloor)$ vertices incident with an edge in $X_s(h)$ and only $h$ of these are $s$-covered. Now $2(h -\lfloor\frac{h-1}{k+1}\rfloor)>h$ because $k \geq 3$ and hence there is a vertex $v$ that is incident with an edge in $X_s(h)$ but is not $s$-covered. Let
\[T=\{j \in \Z_{k+1} \setminus \{s\}:\hbox{an edge in $E(H_j) \setminus X_j(h)$ is incident with $v$}\}\]
and, for each $j \in T$, take $e_j$ to be the edge of $E(H_j) \setminus X_j(h)$ incident with $v$. Then, for each $j \in T \cup \{s\}$, we have that $v$ was not $j$-covered for $X(h)$ but is $j$-covered for $X'(h+1)$ and hence, because at least $h$ vertices were $j$-covered for $X(h)$, at least $h+1$ vertices are $j$-covered for $X'(h+1)$. So we can find $T$ and $X'(h+1)$ with the claimed properties.

It remains to choose $e_j$ for each $j \in \Z_k \setminus (T \cup \{s\})$. If $T \cup \{s\}=\Z_{k+1}$ we are done. Otherwise, let $q$ be an element of $\Z_{k+1} \setminus (T \cup \{s\})$ such that $q+1 \in T \cup \{s\}$. We will show that there is a choice for $e_q$ in $E(H_q) \setminus X_q(h)$ such that at least $h+1$ vertices are $q$-covered for $X'(h+1) \cup \{e_q\}$. This will suffice to complete the proof because a suitable $X(h+1)$ will then be obtainable by iterating this procedure. If more than $h$ vertices are $q$-covered for $X'(h+1)$, then we may take $e_q$ to be an arbitrary edge of $H_q \setminus X_q(h)$. So we can assume that precisely $h$ vertices are $q$-covered for $X'(h+1)$.

If there is a vertex $u$ that is incident with an edge in $E(H_q)\setminus X'_q(h+1)$ but not with an edge in $E(H_{q+1})\setminus X'_{q+1}(h+1)$, then we can take $e_j$ to be the edge in $E(H_q)\setminus X'_q(h+1)$ incident with $u$. The vertex $u$ was not $q$-covered for $X'(h+1)$ but is for $X'(h+1) \cup \{e_q\}$. So it suffices to show that there is such a vertex $u$. Let $w_q$ and $w_{q+1}$ be the number of vertices of $H$ that are incident with no edge in $H_q$ and $H_{q+1}$ respectively. The number of vertices not incident with an edge in $E(H_{q+1})\setminus X'_{q+1}(h+1)$ is thus $2|X'_{q+1}(h+1)|+w_{q+1} \geq 2h -2\lfloor\frac{h}{k+1}\rfloor+w_{q+1}$. The inequality follows because either $q+1=s$ and $|X'_{q+1}(h)|=|X_{q+1}(h)| = h -\lfloor\frac{h-1}{k+1}\rfloor$ by the definition of $s$ or $q+1 \in T$ and $|X'_{q+1}(h)|=|X_{q+1}(h)|+1 \geq h -\lfloor\frac{h-1}{k+1}\rfloor$ because $X(h)$ obeys (i). Now, at most $h$ of these vertices not incident with an edge in $E(H_{q+1})\setminus X'_{q+1}(h+1)$ are not incident with an edge in $X'_q(h+1)$, because precisely $h$ vertices are $q$-covered for $X'(h+1)$, and at most $w_q$ are not incident with an edge in $H_q$. So such a vertex $u$ will exist provided that
\[2h -2\lfloor\tfrac{h}{k+1}\rfloor+w_{q+1}>h+w_q.\]
Now $w_q-w_{q+1} \leq 2$ because $\{H_0,\ldots,H_{k}\}$ is equitable and hence this inequality will hold and such a $u$ will exist unless $h \leq 4$. If $h \leq 4$, then note that, because $\{H_0,\ldots,H_{k}\}$ is equitable, $|E(H_{q+1})| \leq \lceil\tfrac{kn}{2(k+1)}\rceil$ and hence
\[w_{q+1} \geq n-2\lceil\tfrac{kn}{2(k+1)}\rceil > n-\tfrac{kn}{k+1}-2 \geq 4 \geq h\]
where the second last inequality follows because $n \geq 6(k+1)$. Thus, one of the $w_{q+1}$ vertices incident with no edge in $H_{q+1}$ will not be incident with an edge in $X'_q(h+1)$, because precisely $h$ vertices are $q$-covered for $X'(h+1)$. So again such a $u$ exists.

Thus we can choose an $e_q$ in $E(H_q) \setminus X_q(h)$ such that at least $h+1$ vertices are $q$-covered for $X'(h+1) \cup \{e_q\}$. As discussed, by iterating this procedure we can obtain a choice for $X(h+1)$ that satisfies (i), (ii) and (iii). This completes the proof.
\end{proof}

We now present a probabilistic method of finding $(x,w)$-semipartitions of $k$-regular graphs.

\begin{lem}\label{lem:prob cons}
Let $k \geq 3$ be an integer and let $c \in \{k,k+1\}$. If $H$ is a $k$-regular graph of order $n$, $\{H_0,\ldots,H_{c-1}\}$ is an equitable matching decomposition of $H$, and $\alpha$ is a constant such that $0 < \alpha < \frac{k}{c}$, then there is an  $(x,w)$-semipartition of $H$, where
\[
x= \mfrac{\alpha k (2 - \alpha)}{2c}n +O\left(\sqrt{n}\right),\quad w = \mfrac{\alpha k}{c}n+ O\left(\sqrt{n}\right).
\]
\end{lem}

\begin{proof}
For each $v \in V(G)$, let $I_v$ be a random variable that is 1 with probability $\alpha$ and 0 otherwise. Let $R=\{v \in V(G):I_v=1\}$.  Observe that then $|E(H_j[R])|$ is a binomial random variable with $|E(H_j)|$ trials and success probability $\alpha^2$ and so by Hoeffding's inequality \cite{MR144363} we have
\begin{equation}\label{eqn:edgesConcentration}
\mathbb{P}\left(|E(H_j[R])| \leq \alpha^2|E(H_j)|-\sqrt{\tfrac{1}{2}|E(H_j)|\log(4c)}\right) \leq \mfrac{1}{4c}.
\end{equation}
The proof now divides into cases according to whether $c=k$ or $c=k+1$.

\textbf{Case 1.} Suppose that $c=k$. Note that in this case $|E(H_i)|=\frac{n}{2}$
for each $i \in \Z_k$. Now $|R|$ is a binomial random variable with $n$ trials and success probability $\alpha$ and so $\Pr(|R| > \lceil\alpha n\rceil) < \frac{1}{2}$. Thus, by \eqref{eqn:edgesConcentration} and the union bound, there is a subset $S$ of $V$ such that $|S| = \lceil\alpha n\rceil$ and $|E(H_i[S])|  \geq \frac{1}{2}\alpha^2n-O(\sqrt{n})$ for each $i \in \Z_k$ (note that vertices can be added arbitrarily to ensure that $|S| = \lceil\alpha n\rceil$).

Let $j \in \Z_k$. Let $m_j$ be the number of edges of $H_j$ that are incident with at least one vertex in $S$. Because every vertex in $S$ has an edge of $H_j$ incident with it and there are $|E(H_j[S])|$ edges of $H_j$ that are incident with two vertices of $S$, we have
\[m_j=|S|-|E(H_j[S])| \leq \lceil\alpha n\rceil - \tfrac{1}{2}\alpha^2n + O\bigl(\sqrt{n}\bigr) = \alpha n\bigl(1 - \tfrac{\alpha}{2}\bigr) + O\bigl(\sqrt{n}\bigr).\]
So we can take $x=\alpha n(1-\frac{\alpha}{2})+O(\sqrt{n})$ such that $m_i \leq x$ for each $i \in \Z_k$. Then, for each $i \in \Z_k$, we can choose a subset $X_i$ of $E(H_i)$ such that $|X_i|=x$ and $X_i$ contains all $m_i$ edges of $H_i$ that are incident with a vertex in $S$. Because each vertex in $S$ has edges of $H_i$ and $H_{i+1}$ incident with it for each $i \in \Z_k$, $\{X_0,\ldots,X_{k-1}\}$ obeys (P6) for
\[w = |S| = \alpha n + O\left(\sqrt{n}\right)\]
and hence is an $(x,w)$-semipartition with respect to $\{H_0,\ldots,H_{k-1}\}$.

\textbf{Case 2.} Suppose that $c=k+1$. For each $i \in \Z_{k+1}$, note that in this case $|E(H_i)| \in \{\lfloor\frac{kn}{2(k+1)}\rfloor,\lceil\frac{kn}{2(k+1)}\rceil\}$ and let $V_i=V(H) \setminus V(H_i)$. Note that $\{V_0,\ldots, V_{k}\}$ is a partition of $V(H)$ and that $V_i =\frac{n}{k+1}+O(1)$ for each $i \in \Z_{k+1}$. So, for $i \in \Z_{k+1}$, $|R \cap V_i|$ is a binomial random variable with $|V_i|$ trials and success probability $\alpha$ and by Hoeffding's inequality we have
\begin{equation}\label{eqn:verticesConcentration}
\mathbb{P}\left(\bigl||R \cap V_j|-\alpha|V_j|\bigr| \ge \sqrt{\tfrac{1}{2}|V_j|\log(8c)}\right) \leq \mfrac{1}{4c}.
\end{equation}
Thus, by \eqref{eqn:edgesConcentration}, \eqref{eqn:verticesConcentration} and the union bound, there is a subset $S$ of $V$ such that, for each $i \in \Z_{k+1}$, $|S \cap V_i| = \frac{\alpha }{k+1}n+O(\sqrt{n})$ and $|E(H_i[S])| \geq \alpha^2\frac{k}{2(k+1)}n-O(\sqrt{n})$. Note that this implies that $|S|=\alpha n +O(\sqrt{n})$.

Let $j \in \Z_{k+1}$. Let $m_j$ be the number of edges of $H_j$ that are incident with at least one vertex in $S$. Because every vertex in $S \setminus V_j$ has an edge of $H_j$ incident with it and there are $|E(H_j[S])|$ edges of $H_j$ that are incident with two vertices of $S$, we have
\[
m_j=|S|-|S \cap V_j|-|E(H_j[S])| \leq \alpha n - \tfrac{\alpha }{k+1}n - \tfrac{k}{2(k+1)}\alpha^2n + O\bigl(\sqrt{n}\bigr) = \tfrac{\alpha k(2-\alpha)}{2(k+1)}n + O\bigl(\sqrt{n}\bigr).
\]
So we can take $x=\tfrac{\alpha k(2-\alpha)}{2(k+1)}n+O(\sqrt{n})$ such that $m_i \leq x$ for each $i \in \Z_{k+1}$. Then, for each $i \in \Z_{k+1}$, we can choose a subset $X_i$ of $E(H_i)$ such that $|X_i|=x$ and $X_i$ contains all $m_i$ edges of $H_i$ that are incident with a vertex in $S$. Now, for each $i \in \Z_{k+1}$, $|S \setminus V_i|=\frac{\alpha k}{k+1}n+ O(\sqrt{n})$ and each vertex in $S \setminus V_i$ has an edge of $X_i$ incident with it and either has an edge of $X_{i+1}$ incident with it or has no edge of $H_{i+1}$ incident with it. Thus, $\{X_0,\ldots,X_k\}$ obeys (P6) for
\[w=\mfrac{\alpha k}{k+1}n+ O\left(\sqrt{n}\right)\]
and hence is an $(x,w)$-semipartition with respect to $\{H_0,\ldots,H_{k}\}$.
\end{proof}

Now we present the proofs of the lower bounds of Theorem~\ref{thm:k reg}, using the explicit and probabilistic methods for finding $(x,w)$-semipartitions given above. We begin with the following that uses Lemmas~\ref{lem:gen blob construction} and \ref{lem:gen blob construction C2}.
\begin{lem}\label{thm:main thm explicit con}
For any integers $k \geq 3$ and $n \geq 6(k+1)$ such that $nk$ is even,
\[
\cms(n,k) \geq \mfrac{k(5k-3)}{4(k+1)(4k-3)}n-6 \quad \text{ and } \quad  \cms_1(n,k) \geq \mfrac{5k-8}{4(4k-7)}n-6\,.
\]
\end{lem}
\begin{proof}
Let $H$ be a $k$-regular graph with order $n$ and chromatic index $c$.
Let
\[x = \left\lfloor\mfrac{(nk-8c)(c-1)}{2c(4c-7)}\right\rfloor \quad \text{ and } \quad w = x+\left\lfloor \mfrac{x-1}{c-1}\right\rfloor.\]
By Lemma~\ref{lem:gen blob construction} or \ref{lem:gen blob construction C2} there exists
an $(x,w)$-semipartition for $H$, noting that $x \leq \lfloor\frac{nk}{2c}\rfloor$ because $c-1 \leq 4c-7$.
We show that there exists an $(x,w)$-partition. Obviously $w \leq \frac{3}{2}x$, so by Lemma~\ref{lem:can make T with S} it suffices to show
that $x+2y \leq \frac{nk}{2c} $ for
$y = \left\lceil 3x-\frac{3}{2}w\right\rceil$ .

We have
\[
x+2y \leq \mfrac{x(4c-7)}{c-1} +4 \leq \mfrac{nk-8c}{2c} +4 = \mfrac{nk}{2c}
\]
where the first inequality follows because $2y \leq 6x-3w+1$ and $w \geq x+\frac{x-c+1}{c-1}$, and the second follows because
$x(4c-7) \leq \frac{1}{2c}(nk-8c)(c-1)$.

Because there exists an $(x,w)$-partition, we have $\cms(H) \geq x+y-1 $, by Proposition~\ref{prop:seq given some conds}. So
\[ \cms(H) \geq x+y-1  \geq \mfrac{(5c-8)x-2c+5}{2(c-1)}  > \mfrac{k(5c-8)}{4c(4c-7)}n-6 \]
where the second inequality follows because $y \geq 3x-\frac{3}{2}w$ and  $w \leq \frac{xc-1}{c-1}$ and the third follows because $x \geq \frac{(nk-8c)(c-1)}{2c(4c-7)}-1$. Substituting $c=k$ and $c \leq k+1$ gives the required bounds for $\cms_1(n,k)$ and $\cms(n,k)$, respectively.
\end{proof}

\begin{lem}\label{thm:prob con}
Let $k \geq 3$ be an integer. Then for integers $n > k$ such that $nk$ is even,
\[\cms(n,k) \geq \mfrac{31k}{98(k+1)}n-o(n) \quad \text{ and } \quad \cms_1(n,k) \geq \mfrac{31}{98}n-o(n)\]
\end{lem}
\begin{proof}
Let $H$ be a $k$-regular graph with order $n$ and chromatic index $c$.
Let $0 < \alpha < \frac{1}{7}$.
Then, by Lemma~\ref{lem:prob cons}, there is an $(x,w)$-semipartition where
$x = \frac{\alpha k (2 - \alpha)}{2c}n +O\left(\sqrt{n}\right)$ and $w = \frac{\alpha k}{c}n+ O\left(\sqrt{n}\right)$.
For sufficiently large $n$, we can apply Lemma~\ref{lem:can make T with S} to obtain an $(x,w)$-partition, since clearly
$w \leq \frac{3}{2} x$ and making substitutions for $x$ and $w$ shows that
\begin{equation}\label{eqn:oneP1cond}
x+2\left\lceil 3x-\tfrac{3}{2}w \right\rceil
 = \mfrac{\alpha k(8-7\alpha)}{2c}n+O(\sqrt{n})
\leq \left\lfloor \mfrac{nk}{2c} \right\rfloor
\end{equation}
where the last inequality follows from $\alpha < \frac{1}{7}$.
Therefore, applying Proposition~\ref{prop:seq given some conds} yields,
\[
\cms(H) \geq x+\left\lceil 3x-\tfrac{3}{2}w\right\rceil-1
=\mfrac{\alpha k(5-4\alpha)}{2c}n+O(\sqrt{n})\,,
\]
where the final estimate is obtained by making the appropriate substitutions for $x$ and $w$.
As the above inequality holds for any $0 < \alpha < \frac{1}{7}$, by substituting $\alpha$ sufficiently close to $\frac{1}{7}$ we obtain the result.
\end{proof}

We can now present the proof of Theorem~\ref{thm:k reg}.

\begin{proof}[\textbf{\textup{Proof of Theorem~\ref{thm:k reg}.}}]
Let $k \geq 3$ be an integer. Then
the lower bounds for $\cms(n,k)$ and $\cms_1(n,k)$ hold for $n \geq 6(k+1)$ by Lemmas~\ref{thm:main thm explicit con} and \ref{thm:prob con}.
The upper bound for $\cms(n,k)$ follows from Theorem~\ref{thm:count exa},
since for all $n \geq 3k+5$ there is a $k$-regular graph of order $n$ with $B_k$ as a subgraph.
Finally, $\cms(G) \leq \frac{n-1}{2}$ for each $k$-regular graph $G$. To see this, note that otherwise there would be an ordering $\ell=(e_0,\ldots,e_{kn/2-1})$ of a $k$-regular graph $G$ of even order $n$ such that each of $e_{0},\ldots, e_{n/2-1}$ and $e_{1},\ldots, e_{n/2}$ form a matching of size $\frac{n}{2}$ in $G$, which is impossible as two matchings of size $\frac{n}{2}$ cannot differ by exactly one edge.
\end{proof}

\section{Conclusion}

The table below gives, for each integer $k \geq 3$, the strongest consequences of Theorem~\ref{thm:k reg} for large~$n$.

\begin{table}[H]
\begin{center}
\setlength{\tabcolsep}{2pt}
\begin{tabular}{rccclc|crcccl}
  $\frac{1}{4}n-6$ & $\leq$ & $\cms(n,3)$ & $\leq$ & $\frac{3}{8}n$ & \qquad&\qquad & $\frac{7}{20}n-6$ & $\leq$ & $\cms_1(n,3)$ & $\leq$ & $\frac{n-1}{2}$ \\[1mm]
  $\frac{17}{65}n-6$ & $\leq$ & $\cms(n,4)$ & $\leq$ & $\frac{2}{5}n$ & \qquad&\qquad & $\frac{1}{3}n-6$ & $\leq$ & $\cms_1(n,4)$ & $\leq$ & $\frac{n-1}{2}$ \\[1mm]
  $\frac{55}{204}n-6$ & $\leq$ & $\cms(n,5)$ & $\leq$ & $\frac{5}{12}n$ & \qquad&\qquad & $\frac{17}{52}n-6$ & $\leq$ & $\cms_1(n,5)$ & $\leq$ & $\frac{n-1}{2}$ \\[1mm]
  $\frac{27}{98}n-6$ & $\leq$ & $\cms(n,6)$ & $\leq$ & $\frac{3}{7}n$ & \qquad&\qquad & $\frac{11}{34}n-6$ & $\leq$ & $\cms_1(n,6)$ & $\leq$ & $\frac{n-1}{2}$ \\[1mm]
  $\frac{7}{25}n-6$ & $\leq$ & $\cms(n,7)$ & $\leq$ & $\frac{7}{16}n$ & \qquad&\qquad & $\frac{9}{28}n-6$ & $\leq$ & $\cms_1(n,7)$ & $\leq$ & $\frac{n-1}{2}$ \\[1mm]
  $\frac{74}{261}n-6$ & $\leq$ & $\cms(n,8)$ & $\leq$ & $\frac{4}{9}n$ & \qquad&\qquad & $\frac{8}{25}n-6$ & $\leq$ & $\cms_1(n,8)$ & $\leq$ & $\frac{n-1}{2}$ \\[1mm]
  $\frac{63}{220}n-6$ & $\leq$ & $\cms(n,9)$ & $\leq$ & $\frac{9}{20}n$ & \qquad&\qquad & $\frac{37}{116}n-6$ & $\leq$ & $\cms_1(n,9)$ & $\leq$ & $\frac{n-1}{2}$ \\[1mm]
  $\frac{235}{814}n-6$ & $\leq$ & $\cms(n,10)$ & $\leq$ & $\frac{5}{11}n$ & \qquad&\qquad & $\frac{7}{22}n-6$ & $\leq$ & $\cms_1(n,10)$ & $\leq$ & $\frac{n-1}{2}$ \\[1mm]
  $\frac{143}{492}n-6$ & $\leq$ & $\cms(n,11)$ & $\leq$ & $\frac{11}{24}n$ & \qquad&\qquad & $\frac{47}{148}n-6$ & $\leq$ & $\cms_1(n,11)$ & $\leq$ & $\frac{n-1}{2}$ \\[1mm]
  $\frac{19}{65}n-6$ & $\leq$ & $\cms(n,12)$ & $\leq$ & $\frac{6}{13}n$ & \qquad&\qquad & $\frac{13}{41}n-6$ & $\leq$ & $\cms_1(n,12)$ & $\leq$ &$\frac{n-1}{2}$ \\[1mm]
  $\frac{403}{1372}n-6$ & $\leq$ & $\cms(n,13)$ & $\leq$ & $\frac{3}{8}n$ & \qquad&\qquad & $\frac{19}{60}n-6$ & $\leq$ & $\cms_1(n,13)$ & $\leq$ & $\frac{n-1}{2}$ \\[1mm]
   &  &  &  &  & \qquad&\qquad & $\frac{31}{98}n-6$ & $\leq$ & $\cms_1(n,14)$ & $\leq$ & $\frac{n-1}{2}$ \\[1mm]\hline
  \rule{0mm}{5mm}$\frac{31k}{98(k+1)}n-o(n)$ & $\leq$ & $\cms(n,k)$ & $\leq$ & $\frac{k}{2(k+1)}n$ & \qquad&\qquad & $\frac{31}{98}n-o(n)$ & $\leq$ & $\cms_1(n,k)$ & $\leq$ & $\frac{n-1}{2}$ \\[1mm]
  \multicolumn{5}{l}{for each $k \geq 14$} & \qquad&\qquad & \multicolumn{5}{l}{for each $k \geq 15$}
\end{tabular}
\caption{Consequences of Theorem~\ref{thm:k reg} for each $k$}\label{T:kconsequences}
\end{center}
\end{table}

We know of no nontrivial upper bounds on $\cms_1(n,k)$. It would be interesting to obtain some or, on the other hand, to prove that $\cms_1(n,k)$ approaches $\frac{n-1}{2}$ as $n$ becomes large. We certainly expect that our upper bounds on $\cms_1(n,k)$ and $\cms(n,k)$ are much closer to the true value than our lower bounds. In particular we pose the following question.

\begin{question}\label{conj:gen k-reg graph}
Let $k \geq 3$ be an integer. For integers $n > k$ such that $nk$ is even, is it the case that
$
\cms(n,k) = \tfrac{kn}{2(k+1)}-o(n)\,?
$
Is it the case that
$
\cms_1(n,k) = \tfrac{n-1}{2}-o(n)\,?
$
\end{question}

\bigskip
\noindent\textbf{Acknowledgments.}
This work was supported by Australian Research Council grants DP150100506 and FT160100048.


\begin{thebibliography}{1}

\bibitem{MR2394738}
B. Alspach,
The wonderful {W}alecki construction,
{\em Bull. Inst. Combin. Appl.}~{\bf 52} (2008), 7--20.

\bibitem{MR2961987}
R.A. Brualdi, K.P. Kiernan, S.A. Meyer and M.W. Schroeder,
Cyclic matching sequencibility of graphs,
{\em Australas.\ J. Combin.}~{\bf 53} (2012), 245--256.

\bibitem{ChiNa}
S. Chiba and Y. Nakano, Remarks on upper and lower bounds for matching sequencibility of graphs, {\em Filomat} {\bf 30} (2016), 2091--2099.

\bibitem{MR2181045}
G.O.H. Katona,
Constructions via {H}amiltonian theorems,
{\em Discrete Math.}~{\bf 303} (2005), 87--103.

\bibitem{MR3301133}
D.L. Kreher, A. Pastine and L. Tollefson,
A note on the cyclic matching sequencibility of graphs,
{\em Australas. J. Combin.} {\bf 61} (2015), 142--146

\bibitem{MR144363}
W. Hoeffding, Probability inequalities for sums of bounded random variables, \textit{J. Amer. Statist. Assoc.} \textbf{58} (1963), 13--30.

\bibitem{Ma}
A. Mammoliti, The $r$-matching sequencibility of complete graphs. {\em Electron. J. Combin.} {\bf 25} (2018), Paper 1.6.

\bibitem{McD}
C.J.H. McDiarmid, The solution of a timetabling problem, {\em J. Inst.\ Math.\ Appl.}
{\bf 9} (1972), 23--34.

   \bibitem{Viz}
V. G. Vizing, On an estimate of the chromatic class of a $p$-graph
(in Russian), {\em Diskret Analiz} {\bf 3} (1964), 25--30.



\end{thebibliography}
\end{document}